\documentclass[notitlepage]{scrartcl}
\usepackage[shortlabels]{enumitem}
\usepackage{amsmath}
\usepackage{amssymb}
\usepackage{amsthm}
\usepackage{abstract}
\usepackage{xcolor}
\usepackage{comment}
\usepackage{mathtools}
\usepackage{graphicx}
\usepackage{caption}
\usepackage{subcaption}
\usepackage{float}
\usepackage{enumitem}
\usepackage{hyperref}
\usepackage{stmaryrd} 
\hypersetup{
    colorlinks=true,
    linkcolor=blue,
    filecolor=magenta,      
    urlcolor=cyan
    }
\makeatletter
\newcommand*{\barfix}[2][.175ex]{%
  \mathpalette{\@barfix{#1}}{#2}%
}
\newcommand*{\@barfix}[3]{%
  \vbox{%
    \kern#1\relax
    \hbox{$#2#3\m@th$}%
  }%
}
\makeatother

\renewcommand{\exp}{\text{exp}}

\newcommand{\br}[1]{\llbracket{#1}\rrbracket}

\newtheorem{theorem}{Theorem}
\newtheorem{thm}{Theorem}[section]

\newtheorem{corollary}[thm]{Corollary}
\newtheorem{lemma}[thm]{Lemma}
\newtheorem{proposition}[thm]{Proposition}
\newtheorem{claim}[thm]{Claim}

\newcommand{\footremember}[2]{%
    \footnote{#2}
    \newcounter{#1}
    \setcounter{#1}{\value{footnote}}%
}
\newcommand{\footrecall}[1]{%
    \footnotemark[\value{#1}]%
}

\renewcommand{\Pr}{\mathbb{P}}

\usepackage[toc,page]{appendix}

\usepackage[margin=1in]{geometry}
\begin{document}

\title{\vspace{-2em}Tree tilings in random regular graphs}
\author{%
Sahar Diskin \footremember{alley}{School of Mathematical Sciences, Tel Aviv University, Tel Aviv 6997801, Israel.}%
\and Ilay Hoshen \footrecall{alley}%
\and Maksim Zhukovskii \footremember{alley2}{School of Computer Science, The University of Sheffield, Sheffield S1 4DP, United Kingdom.}
}

\date{}
\maketitle
\begin{abstract}
We show that for every $\epsilon>0$ there exists a sufficiently large $d_0\in \mathbb{N}$ such that for every $d\ge d_0$, \textbf{whp} the random $d$-regular graph $G(n,d)$ contains a $T$-factor for every tree $T$ on at most $(1-\epsilon)d/\ln d$ vertices. This is best possible since, for large enough integer $d$, \textbf{whp} $G(n,d)$ does not contain a $\frac{(1+\epsilon)d}{\ln d}$-star-factor. Our method gives a randomised algorithm which \textbf{whp} finds said $T$-factor and whose expected running time is $O(n^{1+o(1)})$, as well as an efficient deterministic counterpart.
\end{abstract}

\section{Introduction}
Let $G$ be an $n$-vertex graph and $H$ be an $s$-vertex graph. An \textit{$H$-factor} in $G$ is a union of $\lfloor\frac{n}{s}\rfloor$ vertex-disjoint isomorphic copies of $H$ in $G$. 

There has been an extensive study into the threshold of appearance of $H$-factors in the binomial random graph $G(n,p)$. The case where $H = K_2$ corresponds to finding a perfect matching in $G(n, p)$. The sharp threshold for appearence of a perfect matching was established by Erd\H{o}s and R\'enyi \cite{ER66}. 
Early results for general $H$ were obtained by Alon and Yuster~\cite{alon1993threshold} and Ruci\'nski~\cite{R92}; they determined the threshold up to a constant factor for a specific family of graphs and gave bounds for the general case. For the case where $H$ is a tree, \L{}uczak and Ruci\'nski \cite{LR91} characterised `pendant' structures, and proved that in the random graph process (that is, when edges are added one after the other uniformly at random), the hitting time of the appearance of an $H$-factor is the same as the hitting time of the disappearance of these forbidden `pendant' structures. In particular, one is able to infer the precise threshold for the appearance of an $H$-factor in this case. In 2008, Johansson, Kahn and Vu \cite{johansson2008factors} determined the threshold (up to a multiplicative constant) for the existence of an $H$-factor for every strictly-1-balanced\footnote{A graph $H$ is strictly-1-balanced if $\frac{|E(H)|}{|V(H)| - 1} > \frac{|E(J)|}{|V(J)| - 1}$ for every proper subgraph $J \subsetneq H$ with $|V(J)| \ge 2$.} graph $H$ and determined the threshold up to a logarithmic factor for an arbitrary graph $H$. In the case of cliques $K_s$, Heckel (for $s=3$)~\cite{H21} and Riordan (for $s\ge 4$)~\cite{R22} determined the \textit{sharp} threshold for the appearance of an $H$-factor. Recently, a hitting time result for the appearance of a $K_s$-factor was proved~\cite{HKMP}, and the sharp threshold for the appearance of an $H$-factor for every strictly-1-balanced graph $H$ was determined~\cite{BHKMP}.

Much less is known in the case of \textit{random $d$-regular graphs}. The random $d$-regular graph $G(n,d)$ is a graph chosen uniformly at random among all simple $d$-regular graphs on the vertex set $\br{n}:=\{1,\ldots,n\}$ (throughout the paper, we treat $d$ as fixed and consider the asymptotics in $n$). Since, for every pair of integers $d \ge 2$ and $k \ge 3$, the number of cycles of length $k$ in $G(n, d)$ is asymptotically distributed as a Poisson random variable with mean $(d-1)^k/(2k)$ (see \cite{W81}), \textbf{whp}\footnote{With high probability, that is, with probability tending to one as $n$ tends to infinity.} there are $o(n)$ cycles of length $k$ in $G(n, d)$. Thus, we may (and will) restrict our attention to tree factors.

For the case of $H=K_2$, Bollob\'as~\cite{B81} proved that whp there exists an $H$-factor (that is, a perfect matching) in $G(n, d)$ for every $d \ge 3$.
There has been some research on the more general case of stars. Naturally, one cannot hope for a $K_{1,t}$-factor for $t>d$, since the graph is $d$-regular. For $d\ge 3$, using a first moment argument, one can show that \textbf{whp} $G(n,d)$ does not contain a $K_{1,d}$-factor (see \cite[Corollary 2]{assiyatun20063}). Robinson and Wormald \cite{robinson1994almost} showed that for $d\ge 3$, \textbf{whp} $G(n,d)$ contains a Hamilton cycle, and thus a $K_{1,2}$-factor. In a subsequent work, Assiyatun and Wormald \cite{assiyatun20063} showed that for $d\ge 4$, \textbf{whp} $G(n,d)$ contains a $K_{1,3}$-factor. One may then suspect that for any $d\ge 3$, typically $G(n,d)$ contains a $K_{1,d-1}$-factor. However, using first moment calculations, one can show that this is not the case for $d\ge 5$ (see Appendix~\ref{appendix}).

There are then two natural avenues to venture into: first, for sufficiently large $d$, to determine all $k$ such that \textbf{whp} $G(n,d)$ contains a factor of stars on $k$ vertices; second, more ambitiously, one could try to find all trees $T$ for which \textbf{whp} $G(n, d)$ contains a $T$-factor.

Considering a related but slightly different problem, Alon and Wormald \cite{alon2010high} showed that for any $d$-regular graph $G$, there exists an absolute constant $c'$, such that $G$ contains a star-factor, in which every star has at least $c'd/\log d$ vertices (not necessarily all stars are of the same size). We stress that here, and throughout the paper, all logarithms are with respect to the natural basis. They further noted that this is optimal up the choice of the constant $c'>0$. Indeed, the existence of a factor of stars on at least $k$ vertices implies the existence of a dominating set of size at most $\frac{n}{k-1}$, and for any $\epsilon>0$ and sufficiently large $d$, \textbf{whp} the smallest dominating set in $G(n,d)$ is of size at least $\frac{(1-\epsilon)n\log d}{d}$ (see \cite[page 3]{alon2010high}). Let us note here that if one only assumes that $G$ is $d$-regular, then one cannot hope to obtain a factor of stars of size exactly $k$ for any $3\leq k=O(d/\log d)$. Indeed, consider for example a $d$-regular graph $G$ formed by a collection of vertex disjoint copies of $K_{d+1}$ and vertex disjoint copies of complete bipartite graphs $K_{d,d}$. Then, since $\text{gcd}(d+1,2d)\in \{1,2\}$, for any choice of $k>2$, one cannot find a factor of stars of size $k$.

Our first main result shows that typically a random $d$-regular graph $G$ contains a star-factor with the asymptotically \textit{optimal} possible size. In fact, we extend this result to factors of \textit{any} tree (not necessarily a star).

\begin{theorem}\label{th: main}
For every constant $0< \epsilon<1$, there exists a sufficiently large integer $d_0$ such that the following holds for any $d\ge d_0$. \textbf{Whp}, for every tree $T$ on at most $\frac{(1-\epsilon) d}{\log d}$ vertices, there exists a $T$-factor in $G(n,d)$.
\end{theorem}
We note that throughout the paper, we will assume that $|V(G)|$ is divisible by $|V(T)|$, to avoid unnecessary technical details, however all proofs can be directly extended to the general case. Further, we note that we may fix the tree $T$ and show that \textbf{whp} there is a $T$-factor in $G(n,d)$; since there are at most $d^2\cdot 4^{d}$ such trees (see~\cite{Otter}) and $d$ is fixed, by the union bound the statement then holds for every tree $T$.

A detailed sketch of the proof of Theorem \ref{th: main} is presented in Section \ref{s: outline}. Let us briefly recap the main strategy here. We show that \textbf{whp}, there exists a balanced partition of $|V(G)|$ into $|V(T)|$ parts so that, for every pair of parts $V_i,V_j$ where $\{i,j\}\in E(T)$, every vertex in $V_i\cup V_j$ has the number of neighbours in the other part concentrated around the mean $d/|V(T)|$. We say that such a partition is \textit{nice}. We obtain this nice partition through four different applications of the algorithmic version of the Lov\'asz Local Lemma, due to Moser and Tardos~\cite{MT10}. In particular, this allows us to find such a nice partition which is \textit{close} to a random partition; further, this gives us a randomised algorithm to find this partition whose average running time is $\tilde{O}(n)$ \textbf{whp} (see Theorem \ref{th: LLL} and Corollary \ref{cor: lll}). In fact, we show that, in \emph{any} $d$-regular graph, which does not have short cycles close to each other, a fraction of the possible partitions are close to nice partitions. Utilising a description of the distribution of edges in random graphs with specified degree sequences (\cite{gao2023subgraph}, see also \cite[Theorem 2.2]{McKaySurvey}), we conclude that \textbf{whp} almost all partitions of a random regular graph induce multipartite graphs with good expansion properties. This allows us to find a partition with such expansion properties which is also close to a nice partition, ensuring the existence of a perfect matching between pairs of sets. Using an algorithm as in~\cite{chen2022maximum}, we can find these perfect matchings in time $n^{1+o(1)}$. This gives us our second main result.
\begin{theorem}\label{cor: running time}
For every constant $0< \epsilon<1$, there exists a sufficiently large integer $d_0$ such that the following holds for any $d\ge d_0$. There is a randomised algorithm that \textbf{whp} finds a $T$-factor in $G(n,d)$ in expected time $n^{1+o(1)}$, for every tree $T$ on at most $\frac{(1-\epsilon) d}{\log d}$ vertices.
\end{theorem}
Since the events that we consider in our applications of the algorithmic version of the Lov\'asz Local Lemma are determined by $\text{poly}(d)$ random variables over domains of size at most $d$, \cite[Theorem 1.4]{MT10} shows that there exists a \textit{deterministic} algorithm that \textbf{whp} finds these $T$-factors in polynomial in $n$ time in $G(n,d)$. 

Let us make some additional remarks.
\begin{itemize}
    \item It is not hard to verify that our proof follows through for a uniformly chosen graph on $n$ vertices with a given degree sequence, whose degrees lie in the interval $[d,(1+\delta)d]$ for some small $\delta>0$. We believe slight modifications of our technique, specifically in Section~\ref{s: theorems proof}, should allow us to obtain the same result for such graphs whose degrees are between $d$ and $O(d)$.
     \item We stress that in order to show the existence of a perfect matching between relevant sets in the partition, we need our partition to be close to a random partition, and thus the application of the algorithmic version of the Lov\'asz Local Lemma is crucial, even if we do not aim to get Theorem~\ref{cor: running time}.
\end{itemize} 
A possible simplification, which allows using a non-constructive version of the Lov\'asz Local Lemma, is applying `sprinkling' due to the contiguity result from~\cite{Janson:contiguity} instead of applying the direct estimation of probabilities in $G(n,d)$. As soon as a nice partition is obtained, we add independently edges of $G(n,\epsilon' d)$, where $\epsilon'\ll\epsilon$. Although it simplifies the proof, it does not allow deriving Theorem~\ref{cor: running time} and the generalisation to non-regular random graphs with specified degree sequences. Moreover, this does not allow obtaining any probability bounds, in contrast to our approach. Indeed, our proof gives that the probability a random $d$-regular graph has a $T$-factor (for any tree $T$ with $|V(T)|\le (1-\epsilon)d/\log d$) is at least $1-n^{-\Theta_d(1)}$.
In fact, the latter probability bound is tight. Indeed, consider the vertices $\{1,...,10d\}\in \br{n}$, say. The probability they form a connected component without a $T$-factor in $G(n,d)$ is at least $n^{-100 d^2}$. We thus obtain the following corollary.
\begin{corollary}
For every constant $0< \epsilon<1$, there exists a sufficiently large integer $d_0$ such that the following holds for any $d\ge d_0$. For any tree $T$ on at most $\frac{(1-\epsilon)d}{\log d}$, the probability that $G(n,d)$ contains a $T$-factor is $1-n^{-\Theta_d(1)}$.
\end{corollary}

One key complication that arises when using any variant of the Lov\'asz Local Lemma to prove Theorem~\ref{th: main} is that it is impossible to directly apply it, as every `bad' event has too many dependencies. A similar issue was addressed independently in the paper by Dragani\'c and Krivelevich~\cite{DK} on connected dominating sets, where they proposed a (significantly different and shorter) proof strategy to show that a $d$-regular graph without short close cycles has a nice partition. Notably, their method requires $\Theta(n)$ applications of the Lov\'asz Local Lemma (and consequently $O(n^2)$ resamples in the algorithmic version), which precludes a linear time reduction to finding perfect matchings. In contrast, our approach applies the Lov\'asz Local Lemma only a constant number of times, enabling such a reduction.

Let us finish this section with several avenues for future research. While in this general setting Theorem~\ref{th: main} is asymptotically best possible, for the case where $T$ is a path on $k$ vertices one can achieve a better result. Indeed, since $G(n,d)$ is typically Hamiltonian \cite{robinson1994almost}, one can \textbf{whp} obtain a factor of paths of any size. In fact, this observation can be generalised to all trees of bounded degree (see below). It would be interesting to try characterising, for every value of $k=k(n)$, families of trees $T$ on $k$ vertices for which one can \textbf{whp} obtain a $T$-factor in $G(n,d)$. 

As mentioned above, our proof uses results \cite{McKaySurvey, gao2023subgraph} on the distribution of edges in graphs chosen uniformly at random given a degree sequence. It would be interesting to see whether this step can be amended to allow our result to hold for pseudo-random $(n,d,\lambda)$-graphs, with $\lambda\ll d$ (see \cite{KS06} for background and many results on pseudo-random graphs). While a slight adjustment of our methods (with, in fact, a much simpler proof) yields that $\left(n,d,O(\sqrt{d})\right)$-graphs contain a star-factor for any star of size at most $\frac{d}{10\log d}$, this could be far from a complete answer. In fact, we are inclined to believe that Theorem \ref{th: main} should hold for $(n,d,\lambda)$-graphs with $\lambda=o(d)$. Let us mention here that, answering a question of Krivelevich \cite{K23}, Pavez-Sign\'e \cite{p23} showed that for $\lambda=o(d)$, an $(n,d,\lambda)$-graph contains a copy of every $n$-vertex tree with bounded degree and $\Theta(n)$ leaves. Subsequent work by Hyde, Morrison, M\"uyesser, and Pavez-Sign\'e \cite{HMMPM23} showed that for $\lambda=o(d/\log^3n)$, an $(n,d,\lambda)$-graph contains a copy of every $n$-vertex tree with bounded degree --- and thus, in particular, contains a $T$-factor for any tree $T$ of bounded degree. 

Another possible direction would be to consider the typical existence of any spanning forests in $G(n,d)$ whose degree is bounded by $(1-\epsilon)d/\log d$. Here, it might be interesting to attempt this first in the model of the binomial random graph, $G(n,p)$, for $p$ above the connectivity threshold, that is, $p\ge \frac{(1+\epsilon)\log n}{n}$. Is it true that it contains any spanning forest $F$ whose degree is bounded by $O(np/\log(np))$ \textbf{whp}? This is naturally tightly related to the universality question, with perhaps one key example being the result of Koml\'os, S\'ark\"{o}zy, and Szemer\'edi~\cite{KSS}, showing that that for every positive $\alpha,\Delta$ and sufficiently large $n$, every graph with minimum degree at least $(1/2+\alpha)n$ contains every tree on $n$ vertices with maximum degree at most $\Delta$.

\subsection{Organisation}
In Section \ref{s: notation} we set out some notation which will be of use throughout the paper. We then discuss the proof's structure and strategy in Section \ref{s: outline}.  In Section \ref{s: prelim} we collect some lemmas which we will utilise in subsequent sections. Section \ref{s: proposition} is devoted to the proof of the key proposition (Proposition \ref{prop: main path}), and is perhaps the most involved and novel part of the paper. Finally, in Section \ref{s: theorems proof} we prove two typical properties of $G(n, d)$ and show how to deduce Theorem~\ref{th: main} from these properties and Proposition \ref{prop: main path}.

\subsection{Notation}\label{s: notation}
Given a graph $H$, a vertex $v\in V(H)$, and a set $A\subseteq V(H)$, we denote by $d_H(v)$ the degree of $v$ and by $d_H(v,A)$ the number of neighbours of $v$ in $A$ (in $H$). When the graph in question is clear we may omit the subscript. We write $d(A)=\sum_{v\in A}d(v)$. Given $A,B\subseteq V(H)$, we denote by $e(A,B)$ the number of edges with one endpoint in $A$ and the other endpoint in $B$. When $A=B$, $e(A)=e(A,A)$ is the number of edges induced by $A$. We denote by $N(A,B)$ the neighbourhood of $A$ in $B$, that is, the set of vertices in $B$ which are adjacent to some vertex in $A$. All logarithms are with the natural base. Moreover, for every positive integer $n$, define $\br{n} \coloneqq \{1, 2, \dots, n\}$. We use the fairly standard notation that given sequences $a=(a_n)$ and  $b=(b_n)\geq 0$, $a=o(b)$ if, for every $\varepsilon>0$ there exists $n_0$ such that $|a_n|\leq\varepsilon b_n$ for all $n\ge n_0$. Given sequences $a'=(a'_d=a'_d(n))$ and $b'=(b'_d=b'_d(n))\geq 0$, we sometimes also use $a=o_d(b)$ to say that, for every $\varepsilon>0$ there exists $d_0,n_0$ such that $|a'_d|\leq\varepsilon b'_d$ for all $d\ge d_0$ and $n\ge n_0$. We systematically ignore rounding signs when it does not affect computations.

\section{Proof outline}\label{s: outline}
Unsurprisingly, finding a tree factor is much harder when the size of the tree is close to the optimal size (that is, $d/\log d$). In this section, we will present the proof outline for Theorem \ref{th: main} in the case when $k\geq\frac{\log d}{10d}$. We will further point out the steps where the proof becomes simpler for trees of smaller size.

Let $T$ be a tree on $k$ vertices and let us label these vertices by $V(T) \coloneqq \br{k}$. The overall strategy for finding a $T$-factor in $G \sim G(n, d)$ is quite intuitive. We will find $k$ disjoint sets $V_1, \dots, V_k \subseteq V(G)$ of equal size and show that \textbf{whp} there exists a perfect matching between every $V_i$ and $V_j$ such that $\{i,j\} \in E(T)$. To do so, our proof proceeds in two main steps. In the first step, we find `good' sets $V_1, \dots, V_k$ (in fact, we show such a partition typically exists in any $d$-regular graph $G$ without two short cycles close to each other). In the second step, we show the typical existence of a perfect matching between every relevant pair of these sets. The properties achieved in the first step facilitate the execution of the second step.

The first step of the proof, presented in Section \ref{s: proposition}, is perhaps the most involved and novel part. In this step, we establish key properties of the sets $V_1,\ldots, V_k$ which will be crucial in verifying the typical existence of perfect matchings in the second step. First, we show that for every $\{i,j\}\in E(T)$, the degree of every $v \in V_i$ into $V_j$ is around $d/k$. Notice that, this property alone does not suffice to establish the existence of a perfect matching between $V_i$ and $V_j$. To that end, we will also make sure that the sets $V_1, \dots, V_k$ are `close' to uniformly chosen sets. 

In the second step, presented in Section \ref{s: theorems proof}, we show that \textbf{whp}, for every $\{i,j\}\in E(T)$, there exists a perfect matching between $V_i$ and $V_j$ by showing that Hall's condition is satisfied. That is, we will show that \textbf{whp}, for every $W \subseteq V_i$, we have $|N(W, V_j)| \ge |W|$. To that end, we utilise a useful bound on the distribution of edges in graphs chosen uniformly at random given a degree sequence (see Theorem~\ref{th:gao} in Section \ref{s: prelim}, see also \cite{McKaySurvey}). First, we will show that \textbf{whp} for every two `small' sets $U,W\subseteq V(G)$ with $|U|=|W|$, there are not too many edges going from $U$ to $W$. Moreover, since the sets $V_1,\ldots, V_k$ were constructed in the first step in a way such that the degree of every vertex $v \in V_i$ to appropriate $V_j$'s is not too small, we obtain a lower bound on $e(W, V_j)$ for every $W \subseteq V_i$. In particular, if $|N(W, V_j)| < |W|$, we will get a contradiction for small sets $W\subseteq V_i$. In the same spirit, Theorem~\ref{th:gao} together with the properties of the sets $V_1,\ldots, V_k$ allows us to bound $|N(W, V_j)|$ for every `large' $W \subseteq V_i$ \textit{if} the sets $V_1, \dots, V_k$ were chosen uniformly at random. Indeed, given randomly chosen disjoint sets $A$ and $B$ (that is, sets formed without first exposing $G(n,d)$), the graph $G[A\cup B]$, given its degree sequence, has a uniform distribution. Luckily, the first step ensures that the sets $V_1.\ldots, V_k$ behave similarly to uniformly chosen sets.

Let us return to the first step of the proof and describe the broad strategy of showing the typical existence of the `good' sets $V_1, \dots, V_k$. We begin with a random partition of the vertices into $k$ parts, $S_1,\ldots, S_k$. A key tool in establishing the existence of such sets $V_1, \ldots, V_k$ is the Lov\'asz Local Lemma. Since we want our sets to be close to the initial random sets $S_1,\ldots, S_k$ (so that we may later be able to apply Theorem \ref{th:gao}), we will in fact utilise the algorithmic version of the Lov\'asz Local Lemma, due to Moser and Tardos (see Theorem \ref{th: LLL} in Section \ref{s: prelim}). Utilising the algorithmic version of the Lov\'asz Local Lemma, we can show that in each step of the algorithm we resample a small number of random variables assigned to vertices, and thus the initial random sets $S_1, \dots, S_k$ will not be `far' from $V_1,\ldots, V_k$. Let us note here that when applying the algorithmic version of the Lov\'asz Local Lemma, in the initial step of the algorithm one evaluates all `bad' events (requires $\tilde O(n)$ time), and then, at each `resampling' step, one re-evaluates only those $O(1)=O_d(1)$ events that depend on the resampled random variables. Since the expected number of steps of the algorithm of Moser and Tardos is $O(n)$, this gives the overall expected running time $\tilde{O}(n)$.

Now, for every vertex $v \in V(G)$, sample $X_v\in \br{k}$ uniformly at random, and independently for all vertices. For every $i \in \br{k}$, set $S_i \coloneqq \{v \in V(G) \colon X_v = i\}$. Moreover, for every vertex $v \in V(G)$, denote by $B_v$ the event that there exists $\{i, j\} \in E(T)$ such that $v \in S_i$ and $d(v, S_j) \notin [\delta d/k, C d/k]$ where $\delta>0$ is a sufficiently small constant and $C>0$ is a sufficiently large constant. Notice that if $\neg B_v$ occurs for every $v \in V(G)$, then we get the desired bounds on the degrees which is the first key point in the first step.

Note that, for every vertex $v \in V(G)$ and $j \in \br{k}$, we have $d(v, S_j) \sim \text{Bin}\left(d, 1/k\right)$. This distribution is the heart of the obstacle concerning `large' trees. The reason for it is that whenever $k \le \frac{d}{10\log d}$, then the probability that the degree of $v$ into $S_j$, for some $j \in \br{k}$, is not in the interval $[\delta d/k, C d/k]$ is at most $d^{-8}$. Furthermore, the event $B_v$ is determined by $d+1$ random variables $X_u$, and $B_v$ is independent of all but at most $d^2$ other events $B_u$. Therefore, we can apply the Lov\'asz Local Lemma. It is worth noting here that if $k\le \frac{d}{10\log d}$, we may omit the requirement that $\{i, j\} \in E(T)$ in the definition of $B_v$ (see Section \ref{s: small prop}). Then, if for every vertex $v \in V(G)$ we have that $\neg B_v$ holds, then $d(v, S_j) \in [\delta d/k, C d/k]$ for every vertex $v \in V(G)$ and index $j \in \br{k}$. 

However, as $k$ gets closer to $\frac{(1-\epsilon)d}{\log d}$, the probability that $\text{Bin}\left(d, 1/k\right) < \delta d/k$ is not smaller than $d^{-1-\epsilon'}$, for some $\epsilon' > 0$ tending to zero as $\epsilon$ tends to zero. Thus, the treatment of this case is much more delicate and involves several rounds of applications of the algorithmic version of the Lov\'asz Local Lemma, in order to refine the initial random partition. In the rest of this section, we describe these rounds.

Notice that, for every $i \in \br{k}$, the event $B_v$ conditioned on $v \in S_i$ is more likely to occur as the degree of the $i$-th vertex in $T$ gets larger. For this reason, we will treat vertices of small degree and vertices of large degree in $T$ differently (this treatment is in Section \ref{s: high deg}). Assume that $\br{h}$ is the set of vertices of $T$ with `large' degrees. We slightly decrease the probability that $X_v = i$, for every $i \in \br{h}$. Then, after one application of the Lov\'asz Local Lemma we will be able to get rid of vertices which have more than $C d/k$ neighbours into $S_j$, for some $j \in \br{h}$. Next, we consider the neighbourhood of the vertices which have degree less than $\delta d/k$ into $S_j$, for some $j \in \br{h}$. We `resample' the vertices in this neighbourhood outside of $S_1,\ldots, S_h$ into $S_1, \dots, S_h$, that is, we move them into one of the sets $S_1,\ldots, S_h$ uniformly at random. In this way, once again using the Lov\'asz Local Lemma, we will have that $d(v, S_j) \in [\delta d/k, C d/k]$ for every $v \in V(G)$ and $j \in \br{h}$.

Next, in Section \ref{s: low degree}, we partition the remaining vertices among $S_{h+1}, \dots, S_k$. In this third application of the Lov\'asz Local Lemma, we will ensure that after the resampling we have the property that $d(v, S_j) \in [\delta d/k, Cd/k]$ for every vertex $v \in V(G)$ and for all but at most $\epsilon^{-2}$ indices $j \in \br{k}$. At this point, there will still be vertices $v \in S_i$ which have less than $\delta d/k$ neighbours into some $S_j$ where $\{i,j\}\in E(T)$. After the fourth application of the Lov\'asz Local Lemma, we will be able to obtain a partition with no such vertices (that is, $d(v, S_j) \in [\delta d/k, Cd/k]$ for every $v\in S_i$ and $\{i,j\}\in E(T)$).

Finally, in Section \ref{s: equal size}, we adjust the sets $S_1, \dots, S_k$ to be of size $n/k$ each. This is the purpose of the fifth and final round of the Lov\'asz Local Lemma. In this round, we will move vertices from sets of size bigger than $\frac{n}{k}$ to sets of size smaller than $\frac{n}{k}$ in a random manner. We will do so while ensuring the vertices $v$ we move satisfy that $d(v, S_j)\in [\delta d/k, Cd/k]$ for every $j\in \br{k}$. After this round, while keeping the bounds over the degrees of the vertices, we will be able to make each set $S_i$ to be close to $n/k$ up to and additive $n/d^{100}$ error term. Finally, to make the sets exactly of size $\frac{n}{k}$, we introduce a deterministic argument adjusting the sets $S_1,\ldots,S_k$ while changing the degree of every vertex to every set $S_i$ by at most one. We thus obtain the required sets $V_1,\ldots, V_k$.

\section{Preliminaries}\label{s: prelim}
We will make use of the following fairly standard Chernoff-type probabilistic bounds (see, for example, Appendix A in \cite{AS16}).
\begin{lemma}\label{lemma:binomial-bounds}
Let $p_1,\ldots, p_n\in [0,1]$. For every $i\in \br{n}$, let $X_i\sim Bernouli(p_i)$, and set $X=\sum_{i=1}^nX_i$. Then,
\begin{enumerate}
    \item For every $b > 0$, 
    \[
        \Pr(X > b\mathbb{E}[X]) \le \left(\frac{e}{b}\right)^{b\mathbb{E}[X]}.
    \]
    \item For any $\delta\ge 0$,
    \[
        \Pr\left(X\ge (1+\delta)\mathbb{E}[X]\right)\le e^{-\frac{\delta^2\mathbb{E}[X]}{2+\delta}}.
    \]
    \item For any $0\le t \le \mathbb{E}[X]$,
    \[
        \Pr\left(X\le \mathbb{E}[X]-t\right)\le e^{-\frac{t^2}{3\mathbb{E}[X]}}.
    \]
\end{enumerate}
\end{lemma}

We also require the following bound on the lower tail of the Binomial distribution.
\begin{lemma}\label{l: binom lower tail}
For every $\xi>0$ there exists $\delta_0>0$ such that for every $\delta\le \delta_0$ the following holds. Let $t\coloneqq t(\xi,\delta)>0$ be sufficiently large. Suppose that $np\ge (1+\xi)t$ and $p\le \frac{1}{2}$. Then,
\begin{align*}
    \mathbb{P}\left(\text{Bin}\left(n,p\right)\le \delta t\right)\le e^{-(1+2\xi/3)t}.
\end{align*}
\end{lemma}
\begin{proof}
Note that we may assume that $np=(1+\xi)t$. We have that
\begin{align*}
    \mathbb{P}\left(\text{Bin}\left(n, p\right) \le \delta t\right) &= \sum_{i = 0}^{\delta t} \binom{n}{i} p^i \left(1 - p\right)^{n-i} 
    = \left(1 - p\right)^{n} \cdot \sum_{i = 0}^{\delta t} \binom{n}{i} \left(\frac{p}{1 - p}\right)^i.
\end{align*}
For the first term,
\[
    \left(1 - p\right)^{n} \le e^{-pn}= e^{-(1+\xi)t}.
\]
For the second term,
\begin{align*}
\sum_{i = 0}^{\delta t} \binom{n}{i} \left(\frac{p}{1 - p}\right)^i&\le 1+\sum_{i=1}^{\delta t}\left(\frac{enp}{i(1-p)}\right)^i\le 1+\sum_{i=1}^{\delta t}\left(\frac{2e(1+\xi)t}{i}\right)^i.
\end{align*}
Since $\delta$ is sufficiently small, the function $g(x) = \left(\frac{2e(1+\xi)t}{x}\right)^x$ is increasing in $x \in (0, \delta t)$ and thus
\[
    1 + \sum_{i=1}^{\delta t}\left(\frac{2e(1+\xi)t}{i}\right)^i \le \delta t \left(\frac{11(1+\xi)}{\delta }\right)^{\delta t}\le e^{\log\left(12(1+\xi)/\delta\right)\cdot\delta t},
\]
recalling that $t=t(\xi,\delta)$ is sufficiently large. Noting that $\delta \cdot \log\left(12(1+\xi)/\delta\right)\le \xi/3$ for $\delta$ sufficiently small, we have that $\mathbb{P}\left(\text{Bin}\left(n, p\right) \le \delta t\right)\le e^{-(1+2\xi/3)t}$, as required.
\end{proof}

We will also make extensive use of the algorithmic version of the Lov\'asz Local Lemma, due to Moser and Tardos \cite{MT10}. 
\begin{theorem}[Theorem 1.2 of \cite{MT10}, rephrased]\label{th: LLL}
Let $U$ be a finite set. Let $X=(\xi_u)_{u\in U}$ be a tuple of mutually independent random variables. Let $\mathcal{F}$ be a finite set of events determined by $X$. Suppose that there exists $q$ such that for every event $F\in \mathcal{F}$, $\mathbb{P}_X(F)\le q$. Moreover, suppose that every $F\in \mathcal{F}$ depends on at most $\Delta$ other events $F'\in \mathcal{F}$. Suppose that $\beta\in (0,1)$ satisfies $q\le \beta(1-\beta)^{\Delta}$. Then, there exists an evaluation of $X$ which does not satisfy any event in $\mathcal{F}$. 

Furthermore, let $(X_n)_{n\in \mathbb{N}}$ be a sequence of mutually independent copies of $X$, that is, for every $n\in \mathbb{N}$, $X_n\sim X$. Initially, we sample $X_0$ and let $Z_0\coloneqq X_0$. At step $t\ge 1$, we pick one $F\in \mathcal{F}$ satisfied by $Z_{t-1}$  (if one exists) in an arbitrary manner. Then, we consider all $u\in U$ such that this witness $F$ depends on the $u$-th coordinate of $Z_{t-1}$, and set $(Z_t)_u=(X_t)_u$ for all such $u$ and $(Z_t)_u=(Z_{t-1})_u$ for all other $u$. If no such $F \in \mathcal{F}$ exists, the process halts and we set $\tau\coloneqq t$. Then, $\mathbb{E}\left[\tau\right]\le |\mathcal{F}|\frac{\beta}{1-\beta}$.
\end{theorem}

In fact, we will utilise the following corollary.
\begin{corollary}\label{cor: lll}
Let $U$ be a finite set. Let $m\in \mathbb{N}$. Let $X=(\xi_u)_{u\in U}$ be a set of mutually independent random variables, supported on $\br{m}$. Let $S_1,\ldots, S_{m}$ be a partition of $U$ satisfying $S_i=\{u\in U\colon \xi_u=i\}$ for every $i \in \br{m}$. Let $\mathcal{F}$ be a finite set of events determined by $S_1,\ldots, S_m$. Suppose that there exists $q$ such that $\mathbb{P}_X(F)\le q$ for every event $F\in \mathcal{F}$. Moreover, suppose that every $F\in \mathcal{F}$ is determined by at most $\Delta_1$ random variables $\xi_u$, and depends on at most $\Delta_2$ other events $F'\in \mathcal{F}$. Furthermore, suppose that $\beta\in (0,1)$ satisfies that $q\le \beta(1-\beta)^{\Delta_2}$. 
Then, the probability (under the measure of $X$) that there exists a partition of $U$ into $U_1,\ldots, U_m$ which does not satisfy any event in $\mathcal{F}$ and $\big|S_i\triangle U_i\big|\le 2\Delta_1|\mathcal{F}|\frac{\beta}{1-\beta}$ for every $i\in \br{m}$, is at least $\frac{1}{2}$.
\end{corollary}
\begin{proof}
Let $\Sigma\coloneqq (X_n)_{n\in \mathbb{N}}$, let $\tau$ and let $Z=(Z_i)_{i=0}^{\tau}$ be the random variables as in the statement of Theorem \ref{th: LLL}. Let $A$ be the event that $\tau \le 2|\mathcal{F}|\frac{\beta}{1-\beta}$. By Theorem \ref{th: LLL}, $\mathbb{E}_{\Sigma}\left[\tau\right]\le |\mathcal{F}|\frac{\beta}{1-\beta}$, and thus by Markov's inequality 
$\mathbb{P}_{\Sigma}(A)\ge\frac{1}{2}$. 

Now, let $B$ be the desired event that there exists a partition of $U$ into $U_1,\ldots, U_m$ which does not satisfy any event in $\mathcal{F}$, and $\big|S_i\triangle U_i\big|\le 2\Delta_1|\mathcal{F}|\frac{\beta}{1-\beta}$ for every $i\in \br{m}$. Note that $B$ is determined by $X_0$. Further, at every step in the algorithm, the number of $u\in U$ for which $(Z_t)_u\neq (Z_{t-1})_u$ is at most $\Delta_1$. Thus, given that $\neg B$ occurs, the number of steps necessary to find an evaluation that does not satisfy any event in $\mathcal{F}$ is more than $\frac{2\Delta_1|\mathcal{F}|\frac{\beta}{1-\beta}}{\Delta_1}=2|\mathcal{F}|\frac{\beta}{1-\beta}$. Hence, given that $\neg B$ occurs, the probability that the algorithm ran at most $2|\mathcal{F}|\frac{\beta}{1-\beta}$ steps is zero. Therefore,
\begin{align*}
   \frac{1}{2}\le \mathbb{P}_{\Sigma}(A)&=\mathbb{P}_{\Sigma}(A\cap B)+\mathbb{P}_{\Sigma}(A\cap \neg B)\\
    &=\mathbb{P}_{\Sigma}(A\cap B)\le \mathbb{P}_{\Sigma}(B)=\mathbb{P}_{X_0}(B),
\end{align*}
as claimed.
\end{proof}

Finally, let us conclude this section with a few statements on the distribution of edges in random graphs which are uniformly chosen among all graphs with a given degree sequence.
\begin{theorem}[Corollary 8 in \cite{gao2023subgraph}]\label{th:gao}
    Given a degree sequence $\mathbf{d} = (d_1, \dots, d_n)$ such that $d_1 \ge d_2 \ge \dots \ge d_n$, let $G$ be a uniform random graph on vertex set $\br{n}$ where vertex $i$ has degree $d_i$. Set $M \coloneqq \sum_{i=1}^{n} d_i$. If an integer $1 \le \ell \le M/2$ satisfies $\sum_{i=1}^{d_1} d_i = o(M - 2\ell)$, then, for every $S_1, S_2 \subseteq V(G)$,
    \[
        \mathbb{P}\left(e(S_1, S_2) \ge \ell\right) \le \binom{d(S_1)}{\ell} \cdot \frac{(d(S_2))_\ell}{(M/2)_\ell (2 + o(1))^\ell} = \binom{d(S_1)}{\ell} \cdot \frac{\binom{d(S_2)}{\ell}}{\binom{M/2}{\ell} (2 + o(1))^\ell}\,.
    \]
\end{theorem}
We will also use the following corollary.
\begin{corollary}\label{cor: gao}
Under the same setting as in the statement of Theorem \ref{th:gao}, if $d(S_2)\le M/2$, we further have
    \begin{align*}
        \mathbb{P}\left(e(S_1, S_2) \ge \ell\right) &\le \binom{d(S_1)}{\ell} \cdot \left(\frac{d(S_2)}{M (1 + o(1))}\right)^\ell.
    \end{align*}. 
\end{corollary}
\begin{proof}
By Theorem \ref{th:gao}, $\mathbb{P}\left(e(S_1, S_2) \ge \ell\right) \le\binom{d(S_1)}{\ell}\frac{(d(S_2))_\ell}{(M/2)_\ell (2 + o(1))^\ell}$. Since $d(S_2)\le \frac{M}{2}$, we have that $\frac{d(S_2)-i}{M/2-i}\le \frac{d(S_2)}{M/2}$ for every $0\le i \le d(S_2)$. Thus,
\begin{align*}
\binom{d(S_1)}{\ell}\frac{(d(S_2))_\ell}{(M/2)_\ell (2 + o(1))^\ell}&\le \binom{d(S_1)}{\ell}\left(\frac{d(S_2)}{M/2}\right)^{\ell}\cdot \frac{1}{(2(1+o(1)))^{\ell}}\\&=\binom{d(S_1)}{\ell} \cdot \left(\frac{d(S_2)}{M (1 + o(1))}\right)^\ell,
\end{align*}
as required.
\end{proof}

We will further utilise the following lemma.
\begin{lemma}\label{lemma:Gao-usage}
    There exists a sufficiently small constant $\xi>0$ such that the following holds. Given a degree sequence $\mathbf{d} = (d_1, \dots, d_n)$ such that $d_1 \ge d_2 \ge \dots \ge d_n$, let $G$ be a uniform random graph on vertex set $\br{n}$ where vertex $i$ has degree $d_i$. Set $M \coloneqq \sum_{i=1}^{n} d_i$. For every $t\in \br{n}$ and every two disjoint sets $A$ and $B$ satisfying
    \[
        M \ge 8(1-3\xi)t, \quad d(A), d(B) \le 2(1+2\xi)t, \quad \text{and} \quad d_1=o\left(\sqrt{t}\right),
    \]
    we have
    \[
        \Pr\left(e(A, B) \ge (1-2\xi)t\right) \le 0.95^{t}.
    \]
\end{lemma}
\begin{proof}
    Set $\ell \coloneqq (1-2\xi)t$. We have that $M-2\ell=\Omega(t)$, $\sum_{i=1}^{d_1}d_i\le d_1^2=o(t)$, and $d(A),d(B)\le M/2$. We may thus apply Theorem \ref{th:gao}. We have
    \begin{align*}    
        \Pr\left(e(A, B) \ge \ell\right) \le \binom{d(A)}{\ell} \cdot \frac{\binom{d(B)}{\ell}}{\binom{M/2}{\ell} \cdot (2 + o(1))^{\ell}}.
    \end{align*}
    Further,
    \[
        \binom{d(A)}{\ell} \le \binom{2(1+2\xi)t}{\ell} \le 4^{(1+2\xi)t}\le 4^{(1+5\xi)\ell},
    \]
    where the last inequality is true for sufficiently small $\xi$. We note that the same upper bound holds for $\binom{d(B)}{\ell}$. Moreover, 
    \[
        \binom{M/2}{\ell} \ge \binom{4(1-3\xi) \ell}{\ell} =\frac{\left((4-12\xi)\ell\right)!}{(\ell)! \left((3-12\xi)\ell\right)!}.
    \] Thus, by Stirling's approximation, for sufficiently small $\xi$,
    \begin{align*}
        \binom{M/2}{\ell} &\ge \frac{\left(4-12\xi\right)^{(4-12\xi)\ell}}{\left(3-12\xi\right)^{(3-12\xi)\ell}} \cdot \frac{1}{\Theta\left(\sqrt{\ell}\right)} 
        \\ &\ge \frac{4^{4\ell}}{3^{3\ell}} \cdot \frac{1}{\Theta\left(100^{\xi \ell} \sqrt{\ell}\right)} \ge 9.4^{\ell} \cdot \frac{1}{\Theta\left(100^{\xi \ell} \sqrt{\ell}\right)}.
    \end{align*}

    Hence, assuming $\xi$ is sufficiently small,
    \begin{align*}
        \Pr\left(e(A, B) \ge \ell\right) &\le \frac{4^{2(1+5\xi)\ell}}{9.4^{\ell} (2+o(1))^{\ell}} \cdot \Theta\left(100^{\xi \ell} \sqrt{\ell}\right) \le 0.9^{\ell}\le  0.95^{t}.
    \end{align*}
    
\end{proof}

\section{Planting the seeds}\label{s: proposition}
As mentioned in Section \ref{s: outline}, the proof of Theorem \ref{th: main} consists of two main steps. In order to find a $T$-factor in $G(n, d)$, we will partition the vertices of $G(n, d)$ into $|V(T)|$ sets of the same size, each set represents a different vertex in the tree $T$, and find a perfect matching between the $i$-th set and the $j$-th set for every $\{i, j\} \in E(T)$. In this section, we find a partition of the vertices into $|V(T)|$ sets which satisfies two crucial properties, which in turn will allow us to find the desired perfect matchings in the second step in Section \ref{s: theorems proof}.

For every pair of integers $d$ and $n$, let $\mathcal{G}_d$ be the family of all $d$-regular graphs on $n$ vertices, such that there are no two cycles of length at most $10$ at distance less than $10$ from each other. Recall that we treat $d$ as fixed, and consider the asymptotic as $n\to\infty$. Note that \textbf{whp} $G(n,d)\in \mathcal{G}_d$ (see, for example, \cite{W99}). 

The main result of this section, which is the first step in the proof of Theorem \ref{th: main}, is the following.
\begin{proposition}\label{prop: main path}
For every $\epsilon > 0$, there exist a sufficiently small constant $\delta\coloneqq \delta(\epsilon)>0$, a sufficiently large constant $C\coloneqq C(\epsilon)>0$, and a sufficiently large integer $d_0$ such that the following holds for any $d\ge d_0$. 

Let $T$ be a tree on $k \le (1-\epsilon) \frac{d}{\log d}$ vertices. Let $G\in \mathcal{G}_d$ and suppose that $n$ is divisible by $k$. Let $S_1,\ldots, S_k$ be a uniformly random partition of $V(G)$: $\mathbb{P}\left(v\in S_i\right)= \frac{1}{k}$ for every $i\in \br{k}$, $v\in V(G)$ and the random choice of $i\in\br{k}$ for $v\in V(G)$ is performed independently of all the other vertices.

Then, with probability bounded away from zero, there are disjoint sets $V_1,\ldots, V_k\subseteq V(G)$, each of size $\frac{n}{k}$, with the following properties.
\begin{enumerate}[(P\arabic*{})]
    \item $|S_i\triangle V_i|=o_d(n/k)$ for every $i\in \br{k}$. \label{p: close to uniform}
    \item $d(v,V_j)\in \left[\frac{\delta d}{k},\frac{Cd}{k}\right]$ for every $\{i, j\} \in E(T)$ and $v \in V_i$. \label{p: good degree between}
\end{enumerate}
\end{proposition}

The proof of Proposition \ref{prop: main path} is composed of five steps. Let us present here the overview of the proof and the organisation of this section.

In the first step of the proof of Proposition \ref{prop: main path}, appearing in Section \ref{s: high deg}, we take care of the sets $V_i$ under construction which correspond to vertices of high degree in $T$.

In the second step of the proof of Proposition \ref{prop: main path}, appearing in Section \ref{s: low degree}, we construct the remaining sets in the partition of $V(G)$ (that is, the sets corresponding to vertices of low degree in $T$). After this step, for every $\{i, j\} \in E(T)$, we will no longer have vertices in the $i$-th set in the partition whose degree into the $j$-th set is greater than $Cd/k$. However, we may still have a small amount of vertices with degree less than $\delta d/k$. In the third step of the proof of Proposition \ref{prop: main path}, appearing in Section \ref{s: bad vertices}, we get rid of all such vertices. Lastly, in the final step of the proof of Proposition \ref{prop: main path}, appearing in Section \ref{s: equal size}, we will balance the sets of the partition to be of size exactly $n/k$ while we ensure that the requested properties are kept. 

As discussed in Section \ref{s: outline}, the proof is much simpler whenever one assumes $k \le \frac{d}{10\log d}$. Indeed, then some of the above steps may be skipped. In Sections~\ref{s: high deg} through \ref{s: equal size}, we focus on the case where $k \ge \frac{d}{10 \log d}$. In Section \ref{s: small prop}, we provide a proof for the smaller values of $k$.

In Sections~\ref{s: high deg} through \ref{s: equal size}, we let $T$ be a tree on $\br{k}$, and (unless explicitly stated otherwise) we assume that $k\ge \frac{d}{10\log d}$. 

\subsection{Vertices of high degree in the tree}\label{s: high deg} 
In this section, we build the sets in the partition of $V(G)$ that correspond to vertices of high degree in the tree. Let $\beta\coloneqq \beta(\epsilon) > 0$ be a sufficiently small constant. Denote by $H_{deg}(T)$ the set of vertices of $T$ whose degree is at least $d^{1-\beta}$, and let $h\coloneqq |H_{deg}(T)|$, noting that $h < d^{\beta}$, since $k=|V(T)|<d$. Assume WLOG that $\br{h} \subseteq V(T) = \br{k}$ is exactly the set $H_{deg}(T)$. We construct the first $h$ sets of the partition, ensuring that every $v\in V(G)$ will have degree between $\delta d/k$ and $Cd/k$ into each one of these sets.

A key tool here is the algorithmic version of the Lov\'asz Local Lemma. We build the first $h$ sets in two rounds. First, we construct random $h$ sets by assigning each vertex into each one of them with probability $(1-\alpha)/k$ for a suitable choice of $\alpha$. After this sample, \textbf{whp} we will not have vertices with degree larger than $C d/k$ into any of the sets. However, we will have a small amount of vertices of degree smaller than $\delta d / k$ into some of the sets. We denote this set of vertices by $B$. In the next round, we will resample the neighbourhood of $B$ outside of the first $h$ sets in the partition, and put each vertex into each one of the first $h$ sets with an appropriate probability, ensuring the expected size of the sets is $n/k$. As we will see, this probability will be at least $900/k$. This, in turn, will allow us to get rid of vertices with less than $\delta d/k$ neighbours into any of the first $h$ sets in the partition.

The next lemma determines the value of $\alpha$ which should be considered.
\begin{lemma}\label{l:alpha-choice}
There exists $c \in [\epsilon/4, 5]$ such that $\alpha = d^{-c}$ satisfies
\[
        \mathbb{P}\left(\text{Bin}\left(d, \frac{1 - \alpha}{k}\right) \le \delta \log d\right) = \frac{\alpha^2}{d\cdot h}.
\]
\end{lemma}
In relation to Theorem \ref{cor: running time}, we note that for the proof to follow, it is sufficient to $n^{-3}$-approximate $\alpha$. As we know that $c\in [\epsilon/4,5]$, the bisection method   (see, for example, \cite{BF93}) gives us an algorithm of finding an $n^{-3}$-approximation of $\alpha$ within $O_{\epsilon}(\log n)$ steps.  
\begin{proof}
Consider the function $f(\alpha) = \mathbb{P}\left(\text{Bin}\left(d, \frac{1 - \alpha}{k}\right) \le \delta \log d\right) - \frac{\alpha^2}{d \cdot h}$. Let us show that there exists $\epsilon/4 \le c\le 5$ such that $f(d^{-c}) = 0$. Noting that $f(\alpha)$ is continuous, it suffices to show that $f\left(d^{-\epsilon/4}\right)<0$, and that $f\left(d^{-5}\right)>0$.

Let us first show that $f(d^{-\epsilon/4}) < 0$. Note that $$d\cdot \frac{1-d^{-\epsilon/4}}{k}\ge \frac{(1-d^{-\epsilon/4})\log d}{1-\epsilon}\ge (1+\epsilon)\log d,$$
where the first inequality is true by the assumption that $k \le (1-\epsilon) \frac{d}{\log d}$ and the last inequality holds for large enough $d$. Thus, by Lemma \ref{l: binom lower tail}, for $\delta$ sufficiently small with respect to $\epsilon$,
\begin{align*}
    \mathbb{P}\left(\text{Bin}\left(d, \frac{1-d^{-\epsilon/4}}{k}\right) \le \delta \log d\right)\le e^{-(1+2\epsilon/3)\log d}<\left(d^{-\epsilon/4}\right)^2/(dh),
\end{align*}
where the last inequality holds for $\beta>0$ sufficiently small. Thus, we obtain $f\left(d^{-\epsilon / 4}\right) < 0$.
        
Furthermore, for every $\alpha \in [0,1]$,
\begin{align*}
    \mathbb{P}\left(\text{Bin}\left(d, \frac{1 - \alpha}{k}\right) \le \delta \log d\right)& > \mathbb{P}\left(\text{Bin}\left(d, \frac{1}{k}\right)=0\right) =\left(1 - \frac{1}{k}\right)^{d}\\
    &\ge e^{-10.5\log d} = d^{-10.5},
\end{align*}
where the last inequality uses that $k\ge \frac{d}{10\log d}$. Thus, we have
\[
    \mathbb{P}\left(\text{Bin}\left(d, \frac{1 - d^{-5}}{k}\right) \le \delta \log d\right)> d^{-10.5}
    >\frac{d^{-2\cdot 5}}{d\cdot h},
\]
and thus $f(d^{-5}) > 0$.

Recalling that  $f(\alpha)$ is a continuous function in $\alpha$, and $f(d^{-\frac{\epsilon}{4}}) < 0 < f(d^{-5})$, there exists $\epsilon/4 \le c \le 5$ such that $f(d^{-c}) = 0$.     
\end{proof}

Throughout the rest of the section, we let $\alpha$ be as in the statement of Lemma \ref{l:alpha-choice}. For every vertex $v \in V(G)$, define the random variable $X_v$ with the following distribution.
\begin{align*}
    \Pr(X_v = i) = \begin{cases}
        \frac{1 - \alpha}{k}, \quad & i \in \br{h} \\
        1-\frac{(1-\alpha)h}{k}, \quad & i = 0.
    \end{cases}
\end{align*}
For every $i \in \br{h}$, let $S_i \coloneqq \{v\in V(G)\colon X_v=i\}$. For every vertex disjoint sets $U_1,\ldots, U_h$, let 
\begin{align*}
    B(U_1,\ldots, U_h)&\coloneqq \left\{v\in V(G)\colon \exists i\in \br{h}, d(v, U_i)\le \delta\log d\right\},\\
    W(U_1,\ldots, U_h)&\coloneqq \left\{v\in V(G)\colon \exists i \in \br{h}, d(v,U_i)\ge C\log d\right\}.
\end{align*}
Set $B_1\coloneqq B(S_1,\ldots, S_h)$ and $W_1\coloneqq W(S_1,\ldots, S_h)$.

Let us first estimate several probabilities that will be useful for us in the proof. 
\begin{lemma}\label{l: probabilities of b1 and w1}
For every $v\in V(G)$, we have the following.
\begin{enumerate}
    \item $\mathbb{P}\left(v\in W_1\right)\le d^{-100}$.
    \item $\mathbb{P}\left(v\in B_1\right)\le \frac{\alpha^2}{d}$.
    \item $\mathbb{P}\left(v\in N(B_1)\setminus \bigcup_{i\in \br{h}}S_i\right)\le 2\left(1-\frac{h(1-\alpha)}{k}\right)\alpha^2$.
\end{enumerate}
\end{lemma}
\begin{proof}
For every $i \in \br{h}$, we have $d(v, S_i) \sim \text{Bin}\left(d,\frac{1-\alpha}{k}\right)$. Using the assumption that $k \ge \frac{d}{10\log d}$, we have $\mathbb{E}[d(v, S_i)] \le 10 \log d$. For the first item of the lemma, by the union bound,
\begin{align*}
    \mathbb{P}\left(v \in W_1\right)\le h\cdot \mathbb{P}\left(\text{Bin}\left(d,\frac{1-\alpha}{k}\right)\ge C\log d\right)\le d^{-100},
\end{align*}
where the last inequality holds by Lemma \ref{lemma:binomial-bounds} for $C$ sufficiently large, recalling that $h\le d^{\beta}$. 

We now turn to the second item. By the union bound and by Lemma \ref{l:alpha-choice},
\begin{align*}
    \mathbb{P}\left(v\in B_1\right)\le h\cdot \mathbb{P}\left(\text{Bin}\left(d,\frac{1-\alpha}{k}\right)\le \delta\log d\right)=\frac{\alpha^2}{d}.
\end{align*}

We are thus left with the third item. Note that
\begin{align*}
    \mathbb{P}\left(v\in N(B_1)\setminus \bigcup_{i\in \br{h}}S_i\right)&=\mathbb{P}\left(v\notin \bigcup_{i\in \br{h}}S_i\right)\mathbb{P}\left(v\in N(B_1)\mid v\notin \bigcup_{i\in \br{h}}S_i\right)\\
    &=\left(1-\frac{h(1-\alpha)}{k}\right)\mathbb{P}\left(v\in N(B_1)\mid v\notin \bigcup_{i\in \br{h}}S_i\right).
\end{align*}
By the union bound 
\begin{align*}
    \mathbb{P}\left(v\in N(B_1)\mid v\notin \bigcup_{i\in \br{h}}S_i\right)&=\mathbb{P}\left(\exists u\in N(v), u\in B_1\mid v\notin \bigcup_{i\in \br{h}}S_i\right)\\
    &\le d\cdot \mathbb{P}\left(u\in B_1\mid u\in N(v), v\notin \bigcup_{i\in \br{h}}S_i\right).
\end{align*}
Once again by the union bound and by Lemma \ref{l:alpha-choice},
\begin{align*}
    \mathbb{P}\left(u\in B_1\mid u\in N(v), v\notin \bigcup_{i\in \br{h}}S_i\right)&\le h\cdot \mathbb{P}\left(\text{Bin}\left(d-1, \frac{1-\alpha}{k}\right)\le \delta \log d\right)\\
    &\le 2h\cdot \mathbb{P}\left(\text{Bin}\left(d, \frac{1-\alpha}{k}\right)\le \delta \log d\right)\\
    &=\frac{2\alpha^2}{d},
\end{align*}
where the second inequality is true by the following. For every integer $i \le \delta \log d$, set $x_i = \binom{d-1}{i} \left(\frac{1-\alpha}{k}\right)^i\left(1 - \frac{1-\alpha}{k}\right)^{d-1-i}$ and $y_i = \binom{d}{i} \left(\frac{1-\alpha}{k}\right)^i\left(1 - \frac{1-\alpha}{k}\right)^{d-i}$. Notice that
\begin{align*}
    \mathbb{P}\left(\text{Bin}\left(d-1, \frac{1-\alpha}{k}\right)\le \delta \log d\right) &= \sum_{i = 0}^{\delta \log d} x_i = \sum_{i = 0}^{\delta \log d} y_i \cdot \frac{d-i}{d\left(1-\frac{1-\alpha}{k}\right)} \\&\le 2 \sum_{i = 0}^{\delta \log d} y_i = 2 \mathbb{P}\left(\text{Bin}\left(d, \frac{1-\alpha}{k}\right)\le \delta \log d\right).
\end{align*}
Thus,
\begin{align*}
       \mathbb{P}\left(v\in N(B_1)\setminus \bigcup_{i\in \br{h}}S_i\right)\le \left(1-\frac{h(1-\alpha)}{k}\right)\cdot2\alpha^2,
\end{align*}
completing the proof.
\end{proof}

We are now ready for the first (out of several) key step in the proof.
\begin{lemma}\label{l: first LLL large degree}
With probability at least $\frac{1}{2}-o(1)$, there exist disjoint sets $A_1^{(1)},\ldots, A_h^{(1)}\subseteq V(G)$ which satisfy the following. Let $B_2=B\left(A_1^{(1)},\ldots, A_h^{(1)}\right)$ and let $W_2=W\left(A_1^{(1)},\ldots, A_h^{(1)}\right)$. Then,
\begin{enumerate}
    \item $\left|S_i\triangle A_i^{(1)}\right|\le \frac{n}{d^{50}}$ and $\left||A_i^{(1)}|-\frac{(1-\alpha)n}{k}\right|\le \frac{n}{d^{50}}$ for every $i\in \br{h}$.
    \item $W_2=\varnothing$.
    \item $\left|N(B_2)\setminus\bigcup_{i\in \br{h}} A_i^{(1)}\right|\le 3\alpha^2n.$
\end{enumerate}
\end{lemma}
\begin{proof}
For every $v \in V(G)$, let $F_v$ be the event that $v\in W_1$. Let $\mathcal{F}\coloneqq \{F_v\}_{v\in V(G)}$. By Lemma~\ref{l: probabilities of b1 and w1}, for every $v\in V(G)$, we have $\mathbb{P}\left(F_v\right)\le d^{-100}\eqqcolon q$. Observe that every $F_v$ is determined by $\Delta_1\coloneqq d$ random variables (its neighbours). Furthermore, every $F_v$ depends on at most $\Delta_2\coloneqq d^2$ other events (revealing whether a vertex $v$ satisfies $F_v$ may only affect the probability that $u$ satisfies $F_u$ for $u$ which is in the second neighbourhood of $v$). Furthermore, note that $\beta\coloneqq 4d^{-100}$ satisfies 
\begin{align}\label{eq: beta q ineq}
    \beta(1-\beta)^{\Delta_2}&=4d^{-100}(1-4d^{-100})^{d^2}
    \ge 4d^{-100}e^{-d^{-90}}
    \ge d^{-100}=q.
\end{align}
By Corollary \ref{cor: lll}, we obtain that with probability at least $\frac{1}{2}$, there exist sets $A_1^{(1)},\ldots, A_h^{(1)}$ such that $W_2=\varnothing$ and $\left|S_i\triangle A_i^{(1)}\right|\le 2\cdot d\cdot n\cdot \frac{4d^{-100}}{1-4d^{-100}}\le \frac{n}{d^{51}}$ for every $i\in \br{h}$. Let $A_1^{(1)},\ldots, A_h^{(1)}$ be these sets (if they do not exist, set $A_i^{(1)}=S_i$ for every $i\in \br{h}$). By Lemma \ref{lemma:binomial-bounds}, \textbf{whp} $\left||S_i|-\frac{(1-\alpha)n}{k}\right|\le n^{2/3}$ for every $i\in \br{h}$, and thus we obtain the first and second items of the lemma. 

As for the third item, by Lemma \ref{l: probabilities of b1 and w1}, $$\mathbb{E}\left[\left|N(B_1)\setminus \bigcup_{i\in \br{h}}S_i\right|\right]\le 2\left(1-\frac{h(1-\alpha)}{k}\right)\alpha^2n.$$ Moreover, the event that $v\in N(B_1)\setminus \bigcup_{i\in \br{h}}S_i$ depends on at most $d^4$ other events $u\in N(B_1)\setminus \bigcup_{i\in \br{h}}S_i$. Thus, 
\begin{align*}
    \text{Var}&\left(\left|N(B_1)\setminus \bigcup_{i\in \br{h}}S_i\right|\right)=\sum_{u,v}\text{Cov}\left(v\in N(B_1)\setminus \bigcup_{i\in \br{h}}S_i, u\in N(B_1)\setminus \bigcup_{i\in \br{h}}S_i\right)\le nd^4.
\end{align*}
Thus, by Chebyshev's inequality, \textbf{whp}
\begin{align}
    \left|N(B_1)\setminus \bigcup_{i\in \br{h}}S_i\right|\le 2\left(1-\frac{h(1-\alpha)}{k}\right)\alpha^2n+n^{2/3}.\label{eq: first}    
\end{align}
Furthermore, whenever we change the location of a vertex (that is, move it from $S_i$ to $S_j$), we change the size of $B_1$ by at most $d$. Hence, 
\begin{align*}
    \big||B_2|-|B_1|\big|\le d\cdot \sum_{i\in \br{h}}\big|S_i\triangle A_i^{(1)}\big|\le d\cdot h\cdot \frac{n}{d^{51}},
\end{align*}
and thus
$
    \left|\left|N\left(B_2\right)\right|-\left|N\left(B_1\right)\right|\right|<\frac{n}{d^{40}}.
$
Together with \eqref{eq: first}, we obtain that
\begin{align*}
    \left|N(B_2)\setminus \bigcup_{i\in \br{h}}A_i^{(1)}\right|\le 3\alpha^2n,
\end{align*}
where we used that $\alpha\ge d^{-5}$.
\end{proof}

Recall that, in order to prove Proposition \ref{prop: main path}, we need to show that the sets $V_1,\ldots, V_k$ exist with positive probability (bounded away from zero). We will actually prove that such sets exist with probability $1/4-o(1)$. In particular, by Lemma \ref{l: first LLL large degree}, the sets $A_1^{(1)},\ldots, A_h^{(1)}$ (satisfying the statement of the lemma) exist with probability $1/2-o(1)$. For every possible tuple of disjoint sets $(S_1,\ldots,S_h)$, if there exists a tuple of sets $(A_1^{(1)},\ldots,A_h^{(1)})$, satisfying the conclusion of Lemma~\ref{l: first LLL large degree}, we fix such a tuple. Otherwise, we let $A_i^{(1)}=S_i$ for all $i\in \br{h}$. Further in this section, we assume that the event from Lemma~\ref{l: first LLL large degree}, that has probability at least $1/2-o(1)$, actually occurs; we call the tuple $(S_1,\ldots,S_h)$ \emph{nice} in this case.
Note that, under this assumption, the sets $A_i^{(1)}$ satisfy that no vertex has more than $C\log d$ neighbours in each one of them.

We turn to show that with probability bounded away from zero, there exist sets $A_1^{(2)},\ldots,A_h^{(2)}$ that are `not far' from $S_1,\ldots, S_h$, and every vertex has between $\delta\log d$ and $2C\log d$ neighbours in each one of $A_1^{(2)},\ldots, A_h^{(2)}$. Let $U$ be a set of size $\frac{\alpha n}{1000}$ which contains $N(B_2)\setminus \bigcup_{i\in \br{h}}A_i^{(1)}$ (note that by Lemma \ref{l: first LLL large degree}, $\left|N(B_2)\setminus \bigcup_{i\in \br{h}}A_i^{(1)}\right|\le 3\alpha^2n<\frac{\alpha n}{1000}$). We will make use of the following lemma.
\begin{lemma}\label{l: sprinkling probability}
There exist $\frac{900}{k}\le p_1,\ldots, p_h\le \frac{1100}{k}$ such that, for every $i\in \br{h}$, 
\begin{align*}
    |A_i^{(1)}|+p_i|U|=\frac{n}{k}.
\end{align*}
\end{lemma}
\begin{proof}
Recall that for every $i\in \br{h}$, $|A_i^{(1)}|\in \left[\frac{(1-\alpha)n}{k}-\frac{n}{d^{50}},\frac{(1-\alpha)n}{k}+\frac{n}{d^{50}}\right]$. Thus, $p_i$ should satisfy
\[
    p_i |U| = \frac{n}{k} - |A_i^{(1)}| \in \left[\frac{\alpha n}{k} - \frac{n}{d^{50}}, \frac{\alpha n}{k} + \frac{n}{d^{50}}\right].
\]
Plugging $|U| = \frac{\alpha n}{1000}$ yields
\[
    p_i \in \left[\frac{1000}{k} - \frac{1000}{\alpha d^{50}}, \frac{1000}{k} + \frac{1000}{\alpha d^{50}}\right] \subseteq \left[\frac{900}{k}, \frac{1100}{k}\right].
\]
\end{proof}

Let $p_1,\ldots, p_h$ be the probabilities from the statement of Lemma \ref{l: sprinkling probability}. Let us note that $\sum_{i\in \br{h}}p_i\le \frac{1100h}{k}<1$, since $h< d^\beta$ and $k\ge \frac{d}{10\log d}$. Further, for every $v\in U$ we define a random variable $X_v$ such that $\mathbb{P}\left(X_v=i\right)=p_i$ for every $i\in \br{h}$. Moreover, for every $i\in \br{h}$, we set $U_i\coloneqq\left\{v\in U\colon X_v=i\right\}$, and let $\tilde{A}_i^{(1)}\coloneqq A_i^{(1)}\cup U_i$. Let $W_3$ be the set of vertices $v\in V(G)$ such that there exists $i\in \br{h}$ for which $d(v,\tilde{A}_i^{(1)})\notin (\delta\log d, 2C\log d)$. 

Let us first bound the probability that $v\in W_3$. Recall that we have already conditioned on the existence of the sets $A_1^{(1)}, \ldots, A_h^{(1)}$ satisfying the properties as in the statement of Lemma~\ref{l: first LLL large degree}. Further, note that the probability measure in the following lemma is induced by the random variables $\{X_v\}_{v \in U}$ given in the previous paragraph.
\begin{lemma}\label{l: W_2 probability}
$\mathbb{P}\left(v\in W_3\right)\le \frac{1}{d^{100}}$ for every $v\in V(G)$.
\end{lemma}
\begin{proof}
We begin with the probability that there exists $i\in \br{h}$ for which $d(v,\tilde{A}_i^{(1)})\ge 2C\log d$. Recall that $d(v, A_i^{(1)})\le C\log d$ for every $i\in \br{h}$. Thus, if there is $i\in \br{h}$ for which $d(v,\tilde{A}_i^{(1)})\ge 2C\log d$, we must have $d(v, U_i)\ge C\log d$. For every $i \in \br{h}$, we have $d(v, U_i) \sim \text{Bin}\left(d(v, U), p_i\right)$. By Lemma \ref{l: sprinkling probability} and by the bound $d(v, U) \le d$,
\begin{align*}
    \mathbb{P}\left(\exists i\in \br{h}, d(v, U_i)\ge C\log d\right)\le h\cdot \mathbb{P}\left(\text{Bin}\left(d,\frac{1100}{k}\right)\ge C\log d\right)<d^{-101},
\end{align*}
where the last inequality is true by Lemma \ref{lemma:binomial-bounds} for $C>0$ large enough.

As for the probability that there exists $i\in \br{h}$ for which $d(v,\tilde{A}_i^{(1)})\le \delta \log d$, note that for every $v \in V(G)$ and $i \in \br{h}$, we have $A_i^{(1)}\subseteq \tilde{A}_i^{(1)}$ and thus $d(v, \tilde{A}_i^{(1)}) \ge d(v, A_i^{(1)})$. Hence, if $v \notin B_2$, then $d(v, \tilde{A}_i^{(1)}) > \delta \log d$ for every $i \in \br{h}$. Further, note that if $v \in B_2$, then since $N(B_2)\setminus \bigcup_{i\in \br{h}} A_i^{(1)}\subseteq U$, we have that $N(v)\setminus \bigcup_{i\in \br{h}} A_i^{(1)}\subseteq U$. Hence,
\begin{align}
    \left|N(v) \cap U\right| = \left|N(v) \setminus \bigcup_{i \in \br{h}} A_i^{(1)}\right| = d - \sum_{i \in \br{h}} d(v, A_i^{(1)}) \ge d - h \cdot C \log d \ge 0.99d. \label{eq: degree y}
\end{align}
Thus, by the union bound and by Lemma \ref{lemma:binomial-bounds}, we have that
\begin{align*}
    \mathbb{P}\left(\exists i \in \br{h}, d(v,\tilde{A}_i^{(1)})\le \delta\log d\right)&\le h\cdot \mathbb{P}(d(v,U_1)\le \delta \log d)\\&\le h\cdot \Pr\left(\text{Bin}\left(0.99d, \frac{900}{k}\right) \le \delta \log d\right) \le d^{-101},
\end{align*}
completing the proof.
\end{proof}

We can now apply Corollary \ref{cor: lll} and obtain the required sets.
\begin{lemma}\label{l: second lll}
With probability at least $1/2-o(1)$, there exist disjoint subsets $A_1^{(2)},\ldots, A_h^{(2)}\subseteq V(G)$ which satisfy the following.
\begin{enumerate}
    \item $\left|S_i\triangle A_i^{(2)}\right|=o_d(n/k)$ for every $i\in \br{h}$.
    \item $\left|A_i^{(2)}-\frac{n}{k}\right|=O(n/d^{50})$ for every $i\in \br{h}$. 
    \item $d(v, A_i^{(2)})\in \left[\delta \log d, 2C\log d\right]$ for every $i\in \br{h}$ and $v\in V(G)$.
\end{enumerate}
\end{lemma}
\begin{proof}
Similarly to the proof of Lemma \ref{l: first LLL large degree}, for every $v\in V(G)$, let $F_v$ be the event that $v\in W_3$. Let $\mathcal{F}\coloneqq \{F_v\}_{v\in V(G)}$. By Lemma \ref{l: W_2 probability}, we have that $\mathbb{P}\left(F_v\right)\le d^{-100}\eqqcolon q$ for every $v\in V(G)$. Observe that every $F_v$ is determined by $\Delta_1\coloneqq d$ random variables. Furthermore, every $F_v$ depends on at most $\Delta_2\coloneqq d^2$ other events. Setting $\beta = 4 d^{-100}$, we have that $\beta(1-\beta)^{\Delta_2} \ge q$ by \eqref{eq: beta q ineq}. Thus, 
by Corollary \ref{cor: lll}, we obtain that with probability at least $\frac{1}{2}$, there exist sets $A_1^{(2)},\ldots, A_h^{(2)}$ such that $d(v, A_i^{(2)})\in \left[\delta \log d, 2C\log d\right]$ for every $i\in \br{h}$ and $v\in V(G)$, obtaining the third item.  Let $A_1^{(2)},\ldots, A_h^{(2)}$ be these sets (if they do not exist, set $A_i^{(2)}=\tilde{A}_i^{(1)}$ for every $i\in \br{h}$).

For the first item, again by Corollary \ref{cor: lll} for every $i\in \br{h}$, 
\begin{align}
    \left|\tilde{A}_i^{(1)}\triangle A_i^{(2)}\right|\le 2\cdot d\cdot n\cdot \frac{4d^{-100}}{1-4d^{-100}}\le \frac{n}{d^{51}}. \label{eq: A2 vs tilde A}    
\end{align}
Recall that $\tilde{A}_i^{(1)}= A_i^{(1)}\cup U_i$. By Lemma \ref{l: first LLL large degree}, $\left|S_i\triangle A_i^{(1)}\right|=o_d(n/k)$, and by Lemma \ref{lemma:binomial-bounds}, \textbf{whp} $|U_i|=o_d(n/k)$ and thus $|\tilde{A}_i^{(1)}\triangle S_i|=o_d(n/k)$.

For the second item, we recall that 
\begin{align}
    \left||A_i^{(1)}|-\frac{(1-\alpha)n}{k}\right|\le \frac{n}{d^{50}} \label{eq: ai vs (1-alpha)n/k}    
\end{align}
for every $i\in \br{h}$. By construction, $|\tilde{A}_i^{(1)}|=|A_i^{(1)}|+|U_i|$. Now, for every $i\in \br{h}$, $$|U_i|\sim \text{Bin}\left(\frac{\alpha n}{1000},p_i\right),$$ where $p_i$ is defined according to Lemma \ref{l: sprinkling probability}. In particular, $p_i$ satisfies that $\left|\mathbb{E}[|U_i|]-\frac{\alpha n}{k}\right|=O(n/d^{50})$. By Lemma \ref{lemma:binomial-bounds}, we have that \textbf{whp} $$|U_i|\in \left[\mathbb{E}[|U_i|]-n^{2/3},\mathbb{E}[|U_i|]+n^{2/3}\right].$$ 
This, together with \eqref{eq: A2 vs tilde A} and \eqref{eq: ai vs (1-alpha)n/k} yields that with probability at least $\frac{1}{2}-o(1)$, we have that $\left||A_i^{(2)}| - \frac{n}{k}\right| = O\left(n/d^{50}\right)$. 
\end{proof}

\subsection{Vertices of low degree in the tree}\label{s: low degree}

We have proved that, if $(S_1,\ldots,S_h)$ is nice (this happens with probability at least $1/2-o(1)$), then there exist sets $A_1^{(2)},\ldots, A_h^{(2)}$ satisfying the properties as in the statement of Lemma \ref{l: second lll}.
Throughout this section, we fix a nice tuple $(S_1,\ldots,S_h)$ and a tuple $(A_1^{(2)},\ldots,A_h^{(2)})$, satisfying the conclusion of Lemma \ref{l: second lll}.


Recall that the sets $A_1^{(2)},\ldots, A_h^{(2)}$ correspond to the high-degree vertices in $T$, and every vertex has between $\delta\log d$ and $2C\log d$ neighbours in each of these sets. As described in the beginning of Section \ref{s: proposition}, in this section we aim to establish similar sets for low-degree vertices.

Let 
\begin{equation}
    U\coloneqq V(G)\setminus \bigcup_{i \in \br{h}} A_i^{(2)}.
    \label{eq:U-def}
\end{equation}
    Note that, by Lemma \ref{l: second lll},
\begin{align}
  |U|&\in \left[n-h\cdot\left(\frac{n}{k}+O\left(\frac{n}{d^{50}}\right)\right), n-h\cdot\left(\frac{n}{k}-O\left(\frac{n}{d^{50}}\right)\right)\right]\nonumber\\
  &=\left[\left(1-\frac{h}{k}-O(d^{-49})\right)n, \left(1-\frac{h}{k}+O(d^{-49})\right)n\right]. \label{eq: size y}
\end{align}
For every $v \in U$, let $X_v$ be the random variable such that $\mathbb{P}\left(X_v=i\right)=\frac{1}{k-h}$, for every index $i\in\{h+1,\ldots,k\}$. All $X_v$ are independent. For every $i\in\{h+1,\ldots,k\}$, set 
\[
    S_i \coloneqq \{v \in U \colon X_v = i\}.
\]
For convenience, let us also set $A_i^{(2)}\coloneqq S_i$ for every $i\in \{h+1,\ldots,k\}$ (recall that $A_i^{(2)}$ is already defined for every $i\in \br{h}$).

Given a partition $U_1,\ldots, U_k$ of $V(G)$, let $B(U_1,\ldots, U_k)$ be the set of vertices $v\in V(G)$ satisfying the following. There exist $i\in \br{k}$ and $j\in \{h+1,\ldots, k\}$ such that $\{i,j\}\in E(T), v\in U_i$ and $d(v,U_j)\le \delta \log d$. Further, let $W(U_1,\ldots, U_k)$ be the set of vertices $v\in V(G)$ that satisfy at least one of the following.
\begin{enumerate}
    \item There exists $i \in \br{k}$ such that $d(v, U_i )\ge 2C \log d$.
    \item There exist more than $1 / \epsilon^2$ indices $i \in \{h+1,\ldots,k\}$ such that $d(v, U_i) < \delta \log d$.
    \item There exists $i \in \br{k}$ such that $d\left(v, U_i \cap B(U_1,\ldots,U_k)\right) > \log \log d$.    
\end{enumerate}
Set $B_4\coloneqq B(A_1^{(2)},\ldots, A_k^{(2)})$ and $W_4\coloneqq W(A_1^{(2)},\ldots, A_k^{(2)})$.

Let us first show that it is quite unlikely for a vertex to be in $W_4$.
\begin{lemma}\label{l: prob w0}
    $\Pr(v \in W_4) \le d^{-100}$ for every vertex $v \in V(G)$.
\end{lemma}
\begin{proof}
    Fix $v \in V(G)$. Since $h \le d^{\beta}$ and $d(v, A_i^{(2)}) \le 2C \log d$, for every $i \in \br{h}$, we have 
    \begin{align}
           d(v, U) \ge d - h \cdot 2C \log d \ge (1 - 0.5\epsilon)d. \label{eq: d(v,U)}
    \end{align}

    For the first item, by the union bound, the probability that there exists $i \in \{h+1,\ldots,k\}$ such that $d(v, A_i^{(2)}) > 2C \log d$ is at most 
    \begin{align*}
        k \cdot \Pr\left(\text{Bin}\left(d(v, U), \frac{1}{k - h}\right) \ge 2C \log d\right)  \le k\cdot \Pr\left(\text{Bin}\left(d,\frac{1}{k-h}\right)\ge 2C\log d\right)< d^{-200},
    \end{align*}
    where the last inequality is true whenever $C$ is large enough by Lemma \ref{lemma:binomial-bounds}.

    For the second item, note that the event that there exist more than $1 / \epsilon^2$ indices $i \in \{h+1,\ldots,k\}$ such that $d(v, A_i^{(2)}) < \delta \log d$ implies that there exist $1 / \epsilon^2$ indices $i_1, \dots, i_{1/\epsilon^2} \in \{h+1,\ldots,k\}$ such that $d\left(v, \bigcup_{j=1}^{1/\epsilon^2} A_{i_j}^{(2)}\right) < \frac{1}{\epsilon^2} \cdot \delta \log d$. Thus, by the union bound over all choices of these indices, this probability is at most $d^{1/\epsilon^2} \Pr\left(\text{Bin}\left(d(v, U), \frac{1}{\epsilon^2(k-h)}\right) < \delta \cdot \frac{\log d}{\epsilon^2}\right).$
    The expectation of this binomial random variable is 
    \begin{align*}
        d(v, U) \cdot \frac{1}{\epsilon^2(k-h)} \ge (1 - 0.5\epsilon)d \cdot \frac{\log d}{\epsilon^2 (1-\epsilon) d} \ge (1+0.5\epsilon) \cdot \frac{\log d}{\epsilon^2}, 
    \end{align*}
    where we used~\eqref{eq: d(v,U)} 
     and $k \le (1-\epsilon)\frac{d}{\log d}$. Thus, for $\delta$ small enough, by Lemma~\ref{l: binom lower tail}
    \begin{align*}
        d^{1/\epsilon^2} \cdot \Pr\left(\text{Bin}\left(d(v, U), \frac{1}{\epsilon^2(k-h)}\right) < \delta \cdot \frac{\log d}{\epsilon^2}\right)
        &\le d^{1/\epsilon^2}\cdot e^{-(1+\epsilon/3)\log d/\epsilon^2} \le d^{-200},
    \end{align*}
    where the last inequality is true whenever $\epsilon$ is small enough.

        For the third item, we fix $i \in \br{k}$ and bound the probability that $d(v, A_i^{(2)} \cap B_4)> \log \log d$ from above. If $i \in \br{h}$, then $d(v, A_i^{(2)}) \le 2C \log d$ and every neighbour $u$ of $v$ in $A_i^{(2)}$ belongs to $B_4$ with probability at most
    \begin{equation}
        (k-h) \cdot \Pr\left(\text{Bin}\left(d(u, U), \frac{1}{k-h}\right) \le \delta \log d\right) \le d^{-\epsilon/2},
    \label{eq:d_u_U}
    \end{equation}
    where the last inequality is true for $\delta$ sufficiently small by Lemma \ref{l: binom lower tail} and \eqref{eq: d(v,U)}. The events that $u_1 \in B_4$ and $u_2 \in B_4$ for two different vertices in $A_i^{(2)} \cap N(v)$ might be dependent if their neighbourhoods intersect. Since $G \in \mathcal{G}_d$, removing at most two vertices from the 2-ball around $v$ ensures that there are no more cycles in the second neighbourhood of $v$. However, this is still not enough to get rid of dependencies since $u_1,u_2$ have the common neighbour~$v$. Nevertheless, after revealing the sets $A_i^{(2)}$ that the vertices of $N(v)$ belong to and deleting at most two vertices of $N(v)$ that belong to a cycle, we can upper bound the events $\{u\in B_4\}$, for all $u\in N(v)\cap A_i^{(2)}$, by independent events that do not consider the vertex $v$. In other words, we bound the probability that $u\in B_4$ from above by $\mathbb{P}\left(\text{Bin}\left(d(u, U)-1, \frac{1}{k-h}\right) \le \delta \log d\right)$, which satisfies the inequality~\eqref{eq:d_u_U} as well. Thus, we obtain that the probability that there are more than $\log \log d$ vertices in $d(v, A_i^{(2)} \cap B_4)$ is at most
    \[
        \Pr\left(\text{Bin}\left(2C \log d, d^{-\epsilon/2}\right) \ge \log \log d - 2\right) \le d^{-\Theta(\log \log d)} \le d^{-200},
    \]
    where the first inequality is true by Lemma \ref{lemma:binomial-bounds}. 
    
    If $i \in \{h+1,\ldots,k\}$, then the probability that a neighbour of $v$ belongs to $A_i^{(2)} \cap B_4$ equals
    \begin{align}
        \Pr\left(u \in A_i^{(2)} \cap B_4 \mid u \in N(v)\right) &= \frac{1}{k - h} \cdot \Pr\left(u \in B_4 \mid u \in N(v) \cap A_i^{(2)}\right)\notag \\
        &\le \frac{1}{k - h} \cdot (k-h) \cdot \Pr\left(\text{Bin}\left(d, \frac{1}{k - h}\right) < \delta \log d\right) \le d^{-1 - \epsilon/2},   
    \label{eq:u_conditional}
    \end{align}
    where 
     the last inequality is true by Lemma \ref{l: binom lower tail} for $\delta$ sufficiently small and since $k-h< (1-\epsilon)d/\log d$. Similar to the previous argument (we again delete a possible cycle crossing $N(v)$, and then bound the events that $u\in B_4$ from above by events that do not consider the vertex $v$ in the neighbourhood of $u$, that is the binomial random variable $\text{Bin}\left(d, \frac{1}{k - h}\right)$ in~\eqref{eq:u_conditional} is replaced by $\text{Bin}\left(d-1, \frac{1}{k - h}\right)$), since $G \in \mathcal{G}_d$, the probability that there are more than $\log \log d$ vertices in $N(v, A_i^{(2)} \cap B_4)$ is at most
    \[
        \Pr\left(\text{Bin}\left(d, d^{-1 - \epsilon/2}\right) \ge \log \log d - 2\right) \le d^{-\Theta(\log \log d)} \le d^{-200},
    \]
    where the first inequality is true by Lemma \ref{lemma:binomial-bounds}. 
\end{proof}

We also require the following lemma.
\begin{lemma}\label{l: typical tilde a}
\textbf{Whp} the following holds.
    \begin{enumerate}
        \item $|B_4| \le n d^{-(1 + \epsilon/4)}$.
        \item For every $i\in \br{k}$, we have $\left||A_{i}^{(2)}|-\frac{n}{k}\right|=O(n/d^{49})$.
        \item For every $i\in \br{k}$, there are at least $\frac{n}{3k}$ vertices $v\in A_i^{(2)}$ such that for every $j\in \br{k}$, $d(v,A_j^{(2)})\ge \delta\log d$.  
    \end{enumerate}
\end{lemma}
\begin{proof}
    We start with the first item. Fix $v \in V(G)$. We consider two cases separately. Let us first assume $v\notin U$. In this case, we bound the event that $v \in B_4$ by the event that there exists $i \in \{h+1, \ldots, k\}$ such that $d(v, A_i^{(2)}) \le \delta \log d$. By Lemma \ref{l: binom lower tail}, for $\delta$ sufficiently small and by~\eqref{eq: d(v,U)},
    \begin{align}
        \Pr\left(v \in B_4 \mid v \notin U\right) \le (k-h) \cdot \Pr\left(\text{Bin}\left(d(v, U), \frac{1}{k-h}\right) < \delta \log d\right) \le d^{-\epsilon/2}. \label{eq: b4 not u}
    \end{align}
    Thus, by \eqref{eq: size y} and for $\beta$ sufficiently small with respect to $\epsilon$ (recalling that $h\le d^{\beta})$, 
    \begin{align*}
        \mathbb{E}[B_4 \setminus U] \le \left(\frac{hn}{k} + O\left(\frac{n}{d^{49}}\right)\right) \cdot d^{-\epsilon/2} \le n d^{-1 - \epsilon/3}.
    \end{align*}
    Now, if $v\in U$, once again by Lemma \ref{l: binom lower tail}, 
    \begin{align*}
        \Pr\left(v \in B_4 \mid v \in U\right) &\le \sum_{i = h+1}^{k} \Pr(v \in A_i^{(2)}\mid v\in U) \sum_{j \in N_T(i) \setminus \br{h}} \Pr\left(d(v, A_j^{(2)}) < \delta \log d \mid v \in A_i^{(2)}\right)
        \\&= \sum_{i = h+1}^{k} \Pr(v \in A_i^{(2)}\mid v\in U) \sum_{j \in N_T(i) \setminus \br{h}} \Pr\left(\text{Bin}\left(d(v, U), \frac{1}{k-h}\right) < \delta \log d\right) \\
        &\stackrel{\eqref{eq: d(v,U)}}\le \sum_{i = h+1}^{k} \frac{1}{k - h} \sum_{j \in N_T(i) \setminus \br{h}} d^{-1 - \epsilon/2} \le 2e(T) \cdot \frac{1}{k-h} \cdot d^{-1 - \epsilon/2} \le d^{-1 - \epsilon/3}.
    \end{align*}
    Hence, we got $\mathbb{E}[|B_4|] \le 2n d^{-1-\epsilon/3}$. Since $G\in \mathcal{G}_d$, we have by the same arguments verbatim as those in the proof of Lemma \ref{l: first LLL large degree} that $\text{Var}(|B_4|) = O(n)$. Thus, by Chebyshev's inequality, \textbf{whp} $|B_4|\le nd^{-(1+\epsilon/4)}$.

    The second item holds by Lemma \ref{l: second lll} for $i\in \br{h}$. For $i\in \{h+1,\ldots,k\}$, we have that $|A_i^{(2)}|$ stochastically dominates $\text{Bin}\left(\left(1-\frac{h}{k}-O(d^{-49})\right)n,\frac{1}{k-h}\right)$, and $|A_i^{(2)}|$ is stochastically dominated by $\text{Bin}\left(\left(1-\frac{h}{k}+O(d^{-49})\right)n,\frac{1}{k-h}\right)$, and thus the second item follows from a simple application of Lemma \ref{lemma:binomial-bounds}. 

   For the third item, given $i\in \br{k}$, denote by $R_i$ the set of vertices $v\in A_i^{(2)}$ such that for every $j\in \br{k}$, $d(v,A_j^{(2)})\ge \delta\log d$. Note that for every $j\in \br{h}$ and $v\in V(G)$, we have that $d(v, A_j^{(2)})\ge \delta\log d$ deterministically. For every $j\in \{h+1,\ldots, k\}$, the probability that $v\in V(G)$ satisfies $d(v,A_j^{(2)})<\delta\log d$ is at most $\mathbb{P}\left(\text{Bin}\left(d(v,U),\frac{1}{k-h}\right)<\delta\log d\right)\le d^{-1-\epsilon/2}$ by Lemma~\ref{l: binom lower tail} together with \eqref{eq: d(v,U)}. Thus, by~\eqref{eq: b4 not u}, $\mathbb{E}[|R_i|]\ge \frac{n}{2k}$. For every $v$, the event that $v\in R_i$ depends on at most $d^4$ other events. Thus, $\text{Var}(|R_i|)=O(n)$, and by Chevyshev's inequality, \textbf{whp} for every $i\in \br{k}$ we have that $|R_i|\ge \frac{n}{3k}$.
\end{proof}

We now turn to use Corollary \ref{cor: lll} to show that there exists a `good' partition $A_1^{(3)},\ldots, A_k^{(3)}$, which is not far from $S_1,\ldots, S_k$.
\begin{lemma}\label{l: third lll}
With probability at least $1/2-o(1)$, there exists a partition of $V(G)$ into $A_1^{(3)},\ldots, A_k^{(3)}$ which satisfies the following. Let 
$$
B_5=B\left(A_1^{(3)},\ldots, A_k^{(3)}\right)
\quad\text{ and }\quad
W_5=W\left(A_1^{(3)},\ldots, A_k^{(3)}\right).
$$
Then,
\begin{enumerate}
    \item $\left|A_i^{(3)}\triangle S_i\right|=o_d(n/k)$ for every $i\in \br{k}$.\label{l: item third lll}
    \item $\left||A_i^{(3)}|-\frac{n}{k}\right|=O(n/d^{49})$ for every $i\in \br{k}$.
    \item $W_5=\varnothing$.
    \item $|B_5|\le nd^{-1-\epsilon/5}$.
    \item For every $i\in \br{k}$, there are at least $\frac{n}{4k}$ vertices $v \in A_i^{(3)}$ which satisfy that for every $j\in \br{k}$, $d(v,A_j^{(3)})\ge \delta\log d$. 
    \item $d(v, A_i^{(3)})\in \left[\delta \log d, 2C\log d\right]$ for every $i\in \br{h}$ and $v\in V(G)$.
\end{enumerate}
\end{lemma}
\begin{proof}
For every $v\in V(G)$, let $F_v$ be the event that $v\in W_5$. Let $\mathcal{F}\coloneqq \{F_v\}_{v\in V}$. By Lemma~\ref{l: prob w0}, we have that for every $v\in V(G)$, $\mathbb{P}\left(F_v\right)\le d^{-100}\eqqcolon q$. Observe that every $F_v$ is determined by at most $\Delta_1\coloneqq 1+d^2$ random variables. Furthermore, every $F_v$ depends on at most $\Delta_2\coloneqq d^4$ other events. Moreover, in the same way as in \eqref{eq: beta q ineq}, $\beta\coloneqq 4d^{-100}$ satisfies $\beta(1-\beta)^{\Delta_2}\ge q$. Thus, by Corollary \ref{cor: lll}, we obtain that with probability at least $\frac{1}{2}$, there exist sets $A_1^{(3)},\ldots, A_k^{(3)}$ such that $W_5=\varnothing$ and 
$$
\left|A_i^{(3)}\triangle A_i^{(2)}\right|\le 2\cdot (1+d^2)\cdot n\cdot \frac{4d^{-100}}{1-4d^{-100}}\le \frac{n}{d^{51}}
$$ 
for every $i\in \br{k}$. Let $A_1^{(3)},\ldots, A_k^{(3)}$ be these sets (if they do not exist, set $A_i^{(3)}=A_i^{(2)}$ for every $i\in \br{k}$). This, together with Lemmas \ref{l: second lll} and \ref{l: typical tilde a}, implies the first three properties. 

We now move to the fourth property. By Lemma \ref{l: typical tilde a}, we have that \textbf{whp} $|B_4|\le nd^{-1-\epsilon/4}$. By construction, \textbf{whp} $|B_5|\le |B_4|+d\sum_{i\in \br{k}}\left|A_i^{(3)}\triangle A_i^{(2)}\right|$. Thus, by the above, we have $|B_5|\le nd^{-1-\epsilon/5}$ \textbf{whp}. 

Similarly, for the fifth property, by Lemma \ref{l: typical tilde a}, we have that \textbf{whp} for every $i\in \br{k}$, there are at least $\frac{n}{3k}$ vertices $v \in A_i^{(2)}$ such that $d(v,A_j^{(2)})\ge \delta\log d$ for every $j\in \br{k}$. Since, by the above, we have moved only $O\left(\frac{n}{d^{50}}\right)$ vertices, we have that \textbf{whp}, for every $i \in \br{k}$, there are at least $\frac{n}{4k}$ vertices $v \in A_i^{(3)}$ such that $d(v,A_j^{(3)})\ge \delta\log d$ for every $j\in \br{k}$. 

For the last property, we simply note that for every $i\in \br{h}$, $A_i^{(3)}=A_i^{(2)}$ and thus this follows from Lemma \ref{l: second lll}.
\end{proof}

\subsection{Eliminating bad vertices}\label{s: bad vertices}

We now say that a tuple $(S_1,\ldots,S_k)$ is \emph{nice} if $(S_1,\ldots,S_h)$ is nice and there exist sets $A_i^{(3)}$, $i\in \br{k}$, satisfying the conclusion of Lemma~\ref{l: third lll}. Due to Lemmas~\ref{l: first LLL large degree} and~\ref{l: third lll}, with probability at least $1/4-o(1)$, the considered random tuple of sets $(S_1,\ldots,S_k)$ is nice. As in the previous section, for every nice tuple $(S_1,\ldots,S_k)$, we fix the corresponding sets $A_i^{(3)}$, $i\in \br{k}$, satisfying the conclusion of Lemma~\ref{l: third lll}. If $(S_1,\ldots,S_k)$ is not nice, then we simply set $A_i^{(3)}=S_i$ for all $i\in \br{k}$.

Further in this section, we assume that the event from Lemma~\ref{l: third lll} (that has probability at least $1/4-o(1)$ due to Lemma~\ref{l: first LLL large degree}), actually occurs; i.e. the tuple $(S_1,\ldots,S_k)$ is \emph{nice}. 
 We recall that $A_1^{(3)},\ldots, A_k^{(3)}$ satisfy that for every vertex $v \in V(G)$ and an index $i \in \br{h}$, we have $d(v, A_i^{(3)}) \in [\delta \log d, 2C \log d]$. Further, for every vertex $v \in V(G)$ and an index $i\in \br{k}$, we have $d(v, A_i^{(3)}) \le 2C \log d$. We now show that after several resamples, we may obtain sets $A_1^{(4)}, \ldots, A_k^{(4)}$ such that for every $\{i, j\} \in E(T)$ and for every $v \in A_i^{(4)}$, the number of neighbours of $v$ in $A_j^{(4)}$ is concentrated around its expectation. More precisely, 
\begin{lemma}\label{l: no b}
There exist sets $A_1^{(4)}, \dots, A_k^{(4)}$ such that the following holds:
    \begin{enumerate}
        \item $\left|A_i^{(4)}\triangle A_i^{(3)}\right|\le nd^{-1-\epsilon/5}$ for every $i\in \br{k}$.\label{l: no b item}
        \item $d(v, A_j^{(4)}) > \frac{\delta \log d}{2}$ for every $\{i,j\} \in E(T)$ and for every $v \in A_i^{(4)}$.
        \item $d(v, A_j^{(4)}) < 3C \log d$ for every vertex $v \in V(G)$ and every $j \in \br{k}$.
        \item For every $i\in \br{k}$, there are at least $\frac{n}{5k}$ vertices $v\in A_i^{(4)}$ which satisfy that for every $j\in \br{k}$, $d(v,A_j^{(4)})\ge \frac{\delta\log d}{2}$. 
    \end{enumerate}
\end{lemma}
\begin{proof}
    For every vertex $v \in B_5$, denote by $I_v \subseteq \{h+1,\ldots,k\}$ the set of indices $j \in \{h+1,\ldots,k\}$ such that $d(v, A_j^{(3)}) \le \delta \log d$. Set $\Gamma_v \coloneqq \{h+1,\ldots,k\} \setminus (I_v \cup N_T(I_v))$. Since $h \le d^{\beta}$ and since for every $i \in \{h+1,\ldots,k\}$ we have $d_T(i) \le d^{1-\beta}$, and recalling that $W_5=\varnothing$,
    \begin{align}
        |\Gamma_v| \ge k - h - \epsilon^{-2} \cdot d^{1-\beta} = (1 - o_d(1)) k.\label{eq: gamma v}
    \end{align}
    For every vertex $v \in B_5$, denote by $X_v$ the uniform random variable over the set of integers $\Gamma_v$. For every $i \in \br{k}$, set
    \[
        \tilde{A}_i^{(3)} \coloneqq (A_i^{(3)} \setminus B_5) \cup \{v \in B_5 \colon X_v = i\}.
    \]
    
    Fix $i\in \br{k}$ and fix $v\in \tilde{A}_i^{(3)}$. Note that for any $j\in \br{k}$,
    \[
        d(v, \tilde{A}_j^{(3)}) \ge d(v, A_j^{(3)}) - d(v, B_5 \cap A_j^{(3)})\ge d(v,A_j^{(3)})-\log\log d,
    \]
    where the second inequality is true by the third item in the definition of $W_5$. We have that at least $\frac{n}{4k}$ vertices in $A_i^{(3)}$ have at least $\delta\log d$ neighbours in $A_j^{(3)}$ for every $j\in \br{k}$. This, together with the above and the bound on $|B_5|$ in Lemma~\ref{l: third lll}, immediately implies the fourth item.
    
    For the second item, fix $\{i,j\} \in E(T)$ and fix a vertex $v \in \tilde{A}_i^{(3)}$. If $v \notin B_5$ or $j \in \br{h}$, then we have $d(v, A_j^{(3)}) > \delta \log d$ and thus
    \[
        d(v, \tilde{A}_j^{(3)}) \ge d(v, A_j^{(3)}) - d(v, B_5 \cap A_j^{(3)}) \ge \delta \log d - \log \log d>\frac{\delta \log d}{2}.
    \]
    Assume now that $v \in B_5$ and $j \in \{h+1,\ldots,k\}$. In this case, we must have that $i \in \Gamma_v$. Hence, for every $\ell \in \br{k}$ such that $\{i, \ell\} \in E(T)$, we have $d(v, A_{\ell}^{(3)}) > \delta \log d$. In particular, $\{i,j\}\in E(T)$. Therefore, we also have
    \[
        d(v, \tilde{A}_j^{(3)}) \ge d(v, A_j^{(3)}) - d(v, B_5 \cap A_j^{(3)}) \ge \delta \log d - \log \log d>\frac{\delta \log d}{2}.
    \]
    Note that the above holds deterministically (that is, independently of the values of all $X_v$). 

    Next, we show that for a fixed vertex $v \in V(G)$, the probability that there exists an index $i \in \br{k}$ such that $d(v, \tilde{A}_i^{(3)}) \ge 3C \log d$ is at most $d^{-199}$. Fix a vertex $v \in V(G)$ and fix an index $i \in \br{k}$. We have $d(v, A_i^{(3)}) < 2C \log d$. Thus, if $d(v, \tilde{A}_i^{(3)}) \ge 3C \log d$, then $d(v,B_5\cap \tilde{A}_i^{(3)})\ge C\log d$. Notice that for every vertex $u \in B_5$, by \eqref{eq: gamma v} we have that $|\Gamma_u| \ge 0.5 k$ and thus the probability that $X_u = i$ (which implies that $u \in \tilde{A}_i^{(3)}$) is at most $2/k$. Therefore,
    \[
        \Pr\left(d(v, B_5 \cap \tilde{A}_i^{(3)}) \ge C \log d\right) \le \Pr\left(\text{Bin}\left(d, 2/k\right) \ge C \log d\right) \le d^{-200},
    \]
    where the last inequality is true whenever $C$ is sufficiently large. Hence, by the union bound over all indices $i \in \br{k}$, we have that the probability that the third item fails for a vertex $v \in V(G)$ is at most $k \cdot d^{-200} \le d^{-199}$.
    For every vertex $v \in V(G)$, denote by $D_v$ the event that the third item in the statement fails for $v$. By the above, $\Pr(D_v) \le d^{-199}$. Moreover, the event $D_v$ is independent on all but at most $d^4$ other events $D_u$. Therefore, by Corollary \ref{cor: lll}, with probability at least $\frac{1}{2}$ there exist sets $A_1^{(4)},\ldots, A_k^{(4)}$ such that none of the events $D_v$ holds. Fixing these sets, since $|B_5|\le nd^{-1-\epsilon/5}$, we have that $|A_i^{(4)}\triangle A_i^{(3)}|\le |B_5|\le nd^{-1-\epsilon/5}$ for every $i\in \br{k}$, completing the proof. 
\end{proof}

\subsection{Balancing the sets}\label{s: equal size} 

We have proved that, if $(S_1,\ldots,S_k)$ is nice (this happens with probability at least $1/4-o(1)$), then there exist sets $A_1^{(4)},\ldots, A_k^{(4)}$ satisfying the properties as in the statement of Lemma \ref{l: no b}.
Throughout this section, we fix a nice tuple $(S_1,\ldots,S_k)$ and a tuple $(A_1^{(4)},\ldots,A_k^{(4)})$, satisfying the conclusion of Lemma \ref{l: no b}.

Recall that the sets $A_1^{(4)},\ldots, A_k^{(4)}$ have good degree distribution in between them, yet their size could be up to $nd^{-1-\epsilon/5}$-far from $n/k$. We now turn to show that there exist sets $A_1^{(5)},\ldots, A_k^{(5)}$, all `close' to $A_1^{(4)},\ldots, A_k^{(4)}$, and all of size $\frac{n}{k}\pm O(nd^{-50})$, which still satisfy the `good degrees' assumption. This will, in turn, allow us to complete the balancing of the sets deterministically, and obtain sets of size exactly $\frac{n}{k}$ which satisfy the `good degrees' assumption. 

To that end, let us reorder the sets such that $A_1^{(4)},\ldots, A_m^{(4)}$ are of size at least $\frac{n}{k}$, and $A_{m+1}^{(4)},\ldots,A_k^{(4)}$ are of size less than $\frac{n}{k}$, for some $m\in \br{k}$. Further, for every $i\in \br{k}$, let $\Delta_i=\left||A_i^{(4)}|-\frac{n}{k}\right|$, noting that by the first item in Lemma \ref{l: no b} and by the second item in Lemma~\ref{l: third lll}, 
\begin{equation}
\Delta_i\le nd^{-1-\epsilon/6}.
\label{eq:Delta_i_upper_bound}
\end{equation} 
We have that for every $i\in \br{k}$, there are at least $\frac{n}{5k}$ vertices $v\in A_i^{(4)}$ such that $d(v,A_j^{(4)})\in \left[\frac{\delta\log d}{2}, 3C\log d\right]$ for every $j\in \br{k}$. For every $i\in \br{m}$, let $Q_i\subseteq A_i(4)$ be a set of exactly $\frac{n}{5k}$ such vertices, and set $Q\coloneqq\bigcup_{i\in \br{m}}Q_i$. 

For every $i \in \br{m}$ and $v \in Q_i$, set $M_v \sim \text{Bernoulli}(p_i)$ where $p_i = \frac{\Delta_i}{n/5k}$. This Bernoulli random variable represents whether the vertex $v$ is moved to the $j$-th set, for some $j \in \{m+1,\ldots, k\}$, or not. In addition, let $Z_v$ be the random variable over the set $\{m+1,\ldots,k\}$ defined by $\Pr(Z_v = j) = \frac{\Delta_j}{\Delta_1 + \ldots  + \Delta_m}$ for every $j \in \{m+1,\ldots, k\}$. Note that $\sum_{j\in\{m+1,\ldots,k\}}\Pr(Z_v=j)=1$ since $\Delta_1+\ldots +  \Delta_m=\Delta_{m+1}+\ldots+\Delta_k$. The random variable $Z_v$ represents the index $j \in \{m+1,\ldots,k\}$ for which the vertex $v$ may move (it will indeed move to $A_j^{(4)}$ if and only if $M_v=1$). We stress that for every $v\in Q$, $M_v$ and $Z_v$ are independent, and are also independent over different $v$. Let
\begin{align*}
    \tilde{A}_i^{(4)} \coloneqq \begin{cases}
        A_i^{(4)} \setminus \{v \in Q_i \colon M_v = 1\}, \quad i \in \br{m} \\
        A_i^{(4)} \cup \{v \in Q \colon M_v = 1 \text{ and } Z_v = i\}, \quad i \in \{m+1,\ldots,k\}.
    \end{cases}
\end{align*}
Note that by the above construction, if $A_i^{(4)}$ is of size smaller than $\frac{n}{k}$, then we may only move vertices into it, whereas when $A_i^{(4)}$ is of size larger than $\frac{n}{k}$ we may only move vertices outside of it. Further, if the set $A_i^{(4)}$ is of size exactly $\frac{n}{k}$, then $\Delta_i=0$, and thus the set will remain unchanged.

We first show some typical properties of the sets $\Tilde{A}_1^{(4)},\ldots, \Tilde{A}_k^{(4)}$, noting that the probability measure here is induced by the random variables $M_v$ and $Z_v$.
\begin{lemma}\label{l: typical aht properties}
\textbf{Whp}, the sets $\tilde{A}_1^{(4)},\ldots, \tilde{A}_k^{(4)}$ satisfy the following for every $i\in \br{k}$.
\begin{enumerate}
    \item $\left|\tilde{A}_i^{(4)}\triangle A_i^{(4)}\right|=o_d(n/k)$.
    \item $\left||\tilde{A}_i^{(4)}|-\frac{n}{k}\right|\le n^{2/3}$.
\end{enumerate}
\end{lemma}
\begin{proof}
Let us first show that, for every $i \in \br{k}$, the expectation of $|\tilde{A}_i^{(4)}|$ is $n/k$. This is equivalent to showing that, for every $i \in \br{k}$, the expected number of vertices which were added/removed is $\Delta_i$. Indeed, if $i \in \br{m}$, then all vertices $v\in Q_i\subseteq A_i^{(4)}$ with $M_v=1$ will no longer be in the $i$-th set. We have
\begin{align*}
    \mathbb{E}[|\{v \in A_i^{(4)} \colon M_v = 1\}|] = \sum_{v \in Q_i} \Pr(M_v = 1) = |Q_i| \cdot \frac{\Delta_i}{n/5k} = \Delta_i.
\end{align*}
For $i \in \{m+1,\ldots,k\}$, all vertices $v\in Q$ with $M_v=1$ and $Z_v=i$ will move to the $i$-th set. We have
\begin{align*}
    \mathbb{E}[|\{v \in V(G) \colon M_v = 1 \text{ and } Z_v = i\}|] &= \sum_{j=1}^{m} \sum_{v \in Q_j} \Pr(M_v = 1 \wedge Z_v = i)\\& = \sum_{j=1}^{m} \sum_{v \in Q_j} \frac{\Delta_j}{n/5k} \cdot \frac{\Delta_i}{\Delta_1 + \dots + \Delta_m} \\
    &= \sum_{j=1}^{m} |Q_j| \cdot \frac{\Delta_j}{n/5k} \cdot \frac{\Delta_i}{\Delta_1 + \dots + \Delta_m} \\&= \sum_{j=1}^{m} \Delta_j \cdot \frac{\Delta_i}{\Delta_1 + \dots + \Delta_m} \\
    &= \Delta_i.
\end{align*}
Altogether, we obtain that, for every $i \in \br{k}$,
\begin{align}
\mathbb{E}\left[\left|\tilde{A}_i^{(4)}\triangle A_i^{(4)}\right|\right] = \Delta_i. \label{eq: deltai expect}
\end{align}

Let us now turn to estimate the probability that $\left|\tilde{A}_i^{(4)}\triangle A_i^{(4)}\right|$ deviates from $\Delta_i$ by more than $n^{2/3}$. Note that $\left|\tilde{A}_i^{(4)}\triangle A_i^{(4)}\right|\sim \text{Bin}(|Q_i|,p_i)$ for every $i\in \br{m}$. For every $i\in \{m+1,\ldots,k\}$, $\left|\tilde{A}_i^{(4)}\triangle A_i^{(4)}\right|$ is distributed according to a sum of $|Q|$ Bernoulli random variables. By \eqref{eq: deltai expect}, the expectation of this sum is $\Delta_i$. Thus, by Lemma \ref{lemma:binomial-bounds}, for any $i\in \br{k}$
\begin{align*}
    \mathbb{P}\left(\left|\tilde{A}_i^{(4)}\triangle A_i^{(4)}\right|\ge \Delta_i+n^{2/3}\right)\le e^{-\frac{n^{1/3}}{4}}.
\end{align*}
Now, if $\Delta_i<n^{2/3}$, then $\mathbb{P}\left(\left|\tilde{A}_i^{(4)}\triangle A_i^{(4)}\right|\le \Delta_i-n^{2/3}\right)=0$ for any $i\in \br{k}$. Otherwise, by Lemma \ref{lemma:binomial-bounds}, for any $i\in \br{k}$
\begin{align*}
\mathbb{P}\left(\left|\tilde A_i^{(4)}\triangle A_i^{(4)}\right|\le \Delta_i-n^{2/3}\right)\le e^{-\frac{n^{1/3}}{4}}.
\end{align*}
Thus 
 the probability that there exists $i\in \br{k}$ such that $\left||\tilde A_1^{(4)}|-\frac{n}{k}\right|>n^{2/3}$ is at most $d\cdot e^{-\frac{n^{1/3}}{4}}=o(1)$. Using that $\Delta_i=o_d(n/k)$ by~\eqref{eq:Delta_i_upper_bound}, we also obtain the first item of this lemma.
\end{proof}

Let $\hat{B}$ be the set of vertices $v\in V(G)$ satisfying at least one of the following.
\begin{itemize}
    \item That there exists $\{i,j\}\in E(T)$ such that $v\in \tilde{A}_i^{(4)}$ and $d(v,\tilde{A}_j^{(4)})\notin \left[\frac{\delta\log d}{3}, 4C\log d\right]$.
    \item $v$ satisfies that $d(v,A_i^{(4)})>\delta \log d/2$ for every $i\in \br{k}$. Further, there is some $i\in \br{k}$ such that $d(v,\tilde{A}_i^{(4)})\le \delta\log d/3$.
\end{itemize} 
\begin{lemma}\label{l: hat B}
For every $v\in V(G)$, $\mathbb{P}\left(v\in \hat{B}\right)\le d^{-100}$.
\end{lemma}
\begin{proof}
For every $j\in \{m+1,\ldots, k\}$, let $B_+(v,j)\coloneqq \{u\in N(v)\cap Q : M_u=1, Z_u=j\}$. That is, $B_+(v,j)$ is the set of neighbours of $v$ which are moved to the $j$-th set. Note that $|B_+(v,j)|$ is the sum of independent Bernoulli random variables. Recalling that $d(v,A_i^{(4)})\leq 3C\log d$ for every $i\in\br{k}$, we have that
\begin{align*}
    \mathbb{E}\left[|B_+(v,j)|\right]&=\sum_{i=1}^m\sum_{u\in N(v)\cap Q_i}\mathbb{P}(M_u=1,Z_u=j)\le \sum_{i=1}^{m}3C\log d \cdot \frac{\Delta_i}{n/5k}\cdot \frac{\Delta_j}{\Delta_1+\cdots+\Delta_m}\\
    &=\frac{3C\log d\cdot \Delta_j}{n/5k}\le \frac{1}{d^{\epsilon/7}},
\end{align*}
where we used $\Delta_j\le nd^{-1-\epsilon/6}$. Thus, by Lemma \ref{lemma:binomial-bounds},
\begin{align*}
    \mathbb{P}\left(|B_+(v,j)|\ge \log\log d\right)\le d^{-101}.
\end{align*}
For every $j\in \br{m}$, let $B_-(v,j)\coloneqq \{u\in N(v)\cap A_j^{(4)}\cap Q\colon M_u=1\}$. That is, $B_-(v,j)$ is the set of neighbours of $v$ which are moved out of $A_j^{(4)}$. Recalling that $d(v,A_j^{(4)})\leq 3C\log d$, we have that $|B_-(v,j)|$ is stochastically dominated by $\text{Bin}\left(3C\log d,\frac{\Delta_i}{n/5k}\right)$. Thus, by Lemma \ref{lemma:binomial-bounds},
\begin{align*}
    \mathbb{P}\left(|B_-(v,j)|\ge \log\log d\right)\le \mathbb{P}\left(\text{Bin}\left(3C\log d,\frac{1}{d^{\epsilon/7}}\right)\ge \log\log d\right)\le d^{-101},
\end{align*}
where we used that $\Delta_i\le nd^{-1-\epsilon/6}$.

Therefore, for any $v\in V(G)$ and $i\in \br{k}$, 
\begin{align}
\mathbb{P}\left(\left|d(v,\tilde{A}_i^{(4)})-d(v,A_i^{(4)})\right|\ge\log\log d\right)\le 2d^{-101}.
\label{eq:d_tilde_d_difference}
\end{align}
Therefore, for the second item in the definition of $\hat{B}$, if $v$ satisfies that $d(v,A_i^{(4)})>\delta \log d/2$, then the probability that $d(v,\tilde{A}_i^{(4)})\le \delta\log d/3$ is at most $2d^{-101}$. Similarly, for the first item, fix $i\in \br{k}$ and $j$ such that $\{i,j\}\in E(T)$. Fix $v\in \tilde{A}_i^{(4)}$. We have that $d(v, A_j^{(4)})\in \left[\frac{\delta\log d}{2},3C\log d\right]$ for any $j$ such that $\{i,j\}\in E(T)$. Indeed, if $v\in A_i^{(4)}$, this holds by Lemma \ref{l: no b}. Otherwise, $v\in Q$ and then for every $j\in \br{k}$ we have $d(v,A_j^{(4)})\in \left[\frac{\delta\log d}{2},3C\log d\right]$. Similarly, the probability that $v\in \hat{B}$ is then at most $2d^{-101}$ due to~\eqref{eq:d_tilde_d_difference}.
\end{proof}

We are now ready to apply Corollary \ref{cor: lll}.
\begin{lemma}\label{l: final lll}
With probability at least $\frac{1}{2}-o(1)$ (in the product measure induced by $M_v$ and $Z_v$, $v\in Q$), there exist disjoint sets $A_1^{(5)},\ldots, A_k^{(5)}$ such that the following holds.
\begin{enumerate}
    \item For every $\{i,j\}\in E(T)$ and for every $v\in A_i^{(5)}$, we have that $d(v,A_j^{(5)})\in \left[\frac{\delta\log d}{3},4C\log d\right]$.
    \item For every $i\in \br{k}$, we have $\left|A_i^{(5)}\triangle \tilde{A}_i^{(4)}\right|=O\left(\frac{n}{d^{50}}\right)$.\label{l: final lll item distance}
    \item For every $i\in \br{k}$, there are at least $\frac{n}{6k}$ vertices $v\in A_i^{(5)}$ which satisfy that $d(v,A_j^{(5)})\in \left[\frac{\delta\log d}{3},4C\log d\right]$ for every $j \in \br{k}$. 
\end{enumerate}
\end{lemma}
\begin{proof}
For every $v\in V(G)$, let $F_v$ be the event that $v\in \hat{B}$. Let $\mathcal{F}\coloneqq \{F_v\}_{v\in V(G)}$. By Lemma~\ref{l: hat B}, we have that for every $v\in V(G)$, $\mathbb{P}\left(F_v\right)\le d^{-100}\eqqcolon q$. Observe that every $F_v$ is determined by at most $\Delta_1\coloneqq 2d+2$ random variables ($M_u$ and $Z_u$ for every $u\in \{v\cup N(v)\})$. Furthermore, every $W(v)$ depends on at most $\Delta_2\coloneqq 2d^4$ other events. Furthermore, similarly to \eqref{eq: beta q ineq}, $\beta\coloneqq 4d^{-100}$ satisfies $\beta(1-\beta)^{\Delta_2}\ge q$. Thus, by Corollary \ref{cor: lll}, we obtain that with probability at least $\frac{1}{2}$, there exist sets $A_1^{(5)},\ldots, A_k^{(5)}$ such that $\hat{B}=\varnothing$ and $\left|A_i^{(5)}\triangle \tilde{A}_i^{(4)}\right|\le 2\cdot (2d+2)\cdot n\cdot \frac{4d^{-100}}{1-4d^{-100}}\le \frac{n}{d^{51}}$ for every $i\in \br{k}$. These sets then satisfy, by definition, the first two items. Fix $A_1^{(5)},\ldots, A_k^{(5)}$ to be these sets (and if they do not exist, set $A_i^{(5)}=\tilde{A}_i^{(4)}$). Then, recall that by Lemma \ref{l: no b}, for every $i\in \br{k}$ there are at least $\frac{n}{5k}$ vertices $v\in A_i^{(4)}$ which satisfy that $d(v,A_j^{(4)})\ge \frac{\delta\log d}{2}$ for every $j\in \br{k}$. Thus, by the second item of $\hat{B}$, the sets $A_1^{(5)},\ldots, A_k^{(5)}$ satisfy that for every $i\in \br{k}$, there are at least $\frac{n}{6k}$ vertices $v\in A_i^{(5)}$ which satisfy that $d(v,A_j^{(5)})\in \left[\frac{\delta\log d}{3},4C\log d\right]$ for every $j \in \br{k}$.
\end{proof}
We have just proved that there exist sets $A_1^{(5)},\ldots, A_k^{(5)}$ satisfying the properties as in the statement of Lemma \ref{l: final lll}. We are now ready to complete the proof of Proposition \ref{prop: main path}. To that end, let us first show we can move the vertices between the sets $A_1^{(5)},\ldots, A_k^{(5)}$ to obtain sets $V_1,\ldots,V_k$, each with exactly $\frac{n}{k}$ vertices, while maintaining the degree distribution between the sets.
\begin{lemma}\label{l: final round}
There exists sets $V_1,\ldots,V_k$ such that the following holds.
\begin{enumerate}
    \item For every $\{i,j\}\in E(T)$ and for every $v\in V_i$, we have that $d(v,V_j)\in \left[\frac{\delta\log d}{4},5C\log d\right]$.
    \item $\left|V_i\triangle A_i^{(5)}\right|=o_d(n/k)$ for every $i\in \br{k}$.\label{l: final item distance}
    \item $|V_i|=\frac{n}{k}$ for every $i\in \br{k}$.
\end{enumerate}
\end{lemma}
\begin{proof}
It suffices to show that we can move vertices from sets of size larger than $\frac{n}{k}$ to sets of smaller size, without changing the degree of any vertex into any of the sets by more than one. Consider the following procedure. We start with the sets $A_1^{(5)},\ldots, A_k^{(5)}$. 

Recall that for every $i\in \br{k}$, at least $\frac{n}{6k}$ vertices $v\in A_i^{(5)}$ satisfy $d(v,A_j^{(5)})\in \left[\frac{\delta\log d}{3},4C\log d\right]$ for every $j\in \br{k}$. Furthermore, by the second item in Lemma \ref{l: typical aht properties} together with the second item in Lemma \ref{l: final lll}, $\left||A_i^{(5)}|-\frac{n}{k}\right|=O(n/d^{50})$ for every $i\in \br{k}$. 

We proceed inductively. Suppose we have already moved vertices $v_1,\ldots, v_t$ and that there still exists a set $A_i^{(5)}$ of size larger than $n/k$. In particular, $t\le k\cdot O(n/d^{50})<n/d^{45}$. Note that
\begin{align*}
    \frac{n}{6k}-d^2t\ge \frac{n}{7k}>0,
\end{align*}
and thus there exists a vertex $v\in A_i^{(5)}$ which satisfies the following two properties:
\begin{itemize}
    \item $v$ is not in the second neighbourhood of any $v_1,\ldots, v_t$, and,
    \item for every $j\in \br{k}$, the number of neighbours of $v$ in $A_j^{(5)}$ lies in the interval $\left[\frac{\delta\log d}{3},5C\log d\right]$.
\end{itemize}
We move the vertex $v$ to an arbitrary set of size smaller than $\frac{n}{k}$. Observe that the above two properties guarantee that throughout the entire process, the degree of any vertex into any set will change by at most one. 
\end{proof}

\begin{proof}[Proof of Proposition \ref{prop: main path}]
We have showed that, with probability $1/4-o(1)$ (in the measure induced by $(S_1,\ldots,S_k)$), there exist sets $V_1,\ldots,V_k$ that satisfy item 1 and item 3 from Lemma~\ref{l: final round} and that $|V_i\triangle S_i|=o_d(n/k)$ for every $i\in \br{k}$, due to Lemmas \ref{l: third lll}\eqref{l: item third lll}, \ref{l: no b}\eqref{l: no b item}, \ref{l: final lll}\eqref{l: final lll item distance}, and \ref{l: final round}\eqref{l: final item distance}. 

Now, let $\tilde{S}_1,\ldots, \tilde{S}_k$ be a uniformly random partition of $V(G)$: every $v\in V(G)$ is assigned to $\tilde{S}_i$ for an index $i\in \br{k}$ chosen uniformly at random, independently of all the other vertices. All that is left then is to observe that there is a coupling $(S'_i,\tilde S_i)$  such that $(S'_1,\ldots,S'_k)\stackrel{d}=(S_1,\ldots,S_k)$ and  $|S'_i\triangle \tilde{S}_i|=o_d(n/k)$ for every $i\in \br{k}$ \textbf{whp}.

Indeed, consider the following coupling. Initially, for every $i\in \br{k}$, we set $S'_i=\tilde{S}_i$. We then keep every vertex $v\in S'_i$, for every $i\in \br{h}$, with probability $1-\alpha$, and with probability $\alpha$ we remove $v$ from $S_i$. Let $N_1$ be the set of removed vertices. Recall that the choice of $\alpha$ is according to Lemma~\ref{l:alpha-choice}, thus we removed at most $n/d^{1+\epsilon/5}$ vertices from every set $S'_i$ \textbf{whp}. 
 Observe that $(S'_1,\ldots,S'_h)\stackrel{d}=(S_1,\ldots,S_h)$ and \textbf{whp} $|S'_i\triangle \tilde{S}_i|=o_d(n/k)$. However, sets $S'_{h+1},\ldots,S'_k$ should be still perturbed since the set $U$ from~\eqref{eq:U-def} is obtained after two resamples due to Corollary~\ref{cor: lll}. Thus, we now consider the first two applications of the algorithmic Lov\'asz Local Lemma. By Lemmas \ref{l: first LLL large degree} and \ref{l: second lll} and since we defined $A_i^{(1)}:=S_i$ for partitions $(S_1,\ldots,S_h)$ that do not satisfy the event from the statement of Lemma \ref{l: first LLL large degree}, for every $i \in \br{h}$, the number of vertices which are moved inside/outside of $S_i$ is $o_d(n/k)$. Denote by $N^{+}_2$ and $N_2^-$ the sets of vertices which were moved in these first two applications of the algorithmic Lov\'asz Local Lemma from outside of $S'_1\cup\ldots\cup S'_h$ to $A_1^{(2)}\cup\ldots\cup A_h^{(2)}$ and from $S'_1\cup\ldots\cup S'_h$ outside of $A_1^{(2)}\cup\ldots\cup A_h^{(2)}$, respectively. We have that $|N_2^-|=o_d(n)$ and $|N_2^+|=o_d(n)$. We then partition the set $N_1\cup N_2^-\setminus N_2^+$ that has size $o_d(n)$ \textbf{whp} uniformly at random into $S''_{h+1},\ldots,S''_k$. Letting $S'_i:=S'_i\cup S''_i$, for every $i\in\br{k}\setminus\br{h}$, and recalling that $k-h=(1-o_d(1))k$, we get that $(S'_1,\ldots,S'_k)\stackrel{d}=(S_1,\ldots,S_k)$ and that $|S'_i\triangle \tilde{S}_i|=o_d(n/k)$ for every $i\in \br{k}$ \textbf{whp}, completing the proof.
\end{proof}

\subsection{Small trees}\label{s: small prop}
For small trees we prove a stronger version of Proposition \ref{prop: main path}:
\begin{proposition}\label{prop: main small k}
Let $k\le \frac{d}{10\log d}$. Let $G\in \mathcal{G}_d$, and suppose that $n$ is divisible by $k$. Then, there exists a sufficiently large constant $C\coloneqq C(\epsilon)>0$ and a sufficiently small constant $\delta\coloneqq \delta(\epsilon)>0$ such that the following holds.

Let $S_1,\ldots, S_k$ be a uniformly random partition of $V(G)$: for every $i\in\br{k}$ and for every $v\in V(G)$, the vertex $v$ belongs to $S_i$ with probability $1/k$, independently of all the other vertices. Then, with probability bounded away from zero, there are disjoint sets $V_1,\ldots, V_k\subseteq V(G)$, each of size $\frac{n}{k}$, with the following properties.
\begin{enumerate}[(P\arabic*{})]
    \item $|S_i\triangle V_i|=o_d(n/k)$ for every $i\in \br{k}$. \label{p: close to uniform small}
    \item $d(v,V_i)\in \left[\frac{\delta d}{k},\frac{Cd}{k}\right]$ for every $i\in \br{k}$ and $v\in V(G)$. \label{p: good degree between small}
\end{enumerate}
\end{proposition}
\begin{proof}[Proof]
We begin by assigning to every vertex $v\in V(G)$ a random variable $X_v$ supported on $\br{k}$, where $\mathbb{P}(X_v = i)=\frac{1}{k}$ for every $i \in \br{k}$. For every $i \in \br{k}$, set $S_i \coloneqq \{v\in V(G)\colon X_v=i\}$. Let $W$ be the set of vertices $v \in V(G)$ such that there exists $i\in \br{k}$ for which $d(v,S_i)\notin[\delta d/k, Cd/k]$. Now, by Lemma \ref{lemma:binomial-bounds},
\begin{align*}
    \mathbb{P}\left(\exists i\in \br{k}, d(v,S_i)\ge Cd/k\right)&\le d\cdot \mathbb{P}\left(\text{Bin}\left(d,\frac{1}{k}\right)\ge Cd/k\right)\\
    &\le d\cdot e^{-\frac{Cd}{4k}}\le d^{-100},
\end{align*}
for sufficiently large $C$. Further, 
\begin{align*}
    \mathbb{P}\left(\exists i\in \br{k}, d(v,S_i)\le \delta d/k\right)
    &\le d\cdot \mathbb{P}\left(\text{Bin}\left(d,\frac{1}{k}\right)\le \delta d/k\right)\\
    &\le d^2{\binom{d}{\delta d/k}}k^{-\delta d/k}(1-1/k)^{(1-\delta/k)d}\\
    &\leq d^2\exp\left(\frac{d}{k}(\delta\ln (e/\delta)-1+\delta)\right)\\
    &\leq d^2\cdot d^{-10(1-\delta(1+\ln(e/\delta)))}\le d^{-7},
\end{align*}
where we used that $k\le \frac{d}{10\log d}$. 

For every $v\in V(G)$, let $F_v$ be the event that $v\in W$, and let $\mathcal{F}\coloneqq \{F_v\}_{v\in V(G)}$. Observe that every $F_v$ is determined by $\Delta_1\coloneqq d$ random variables (its neighbours). Furthermore, every $F_v$ depends on at most $\Delta_2\coloneqq d^2$ other events (revealing whether a vertex $v$ satisfies $F_v$ may only affect the probability that $u$ satisfies $F_u$ for $u$ which is in the second neighbourhood of $v$). Furthermore, note that $\beta\coloneqq 4d^{-4}$ satisfies 
\begin{align*}
    \beta(1-\beta)^{\Delta_2}&=4d^{-4}(1-4d^{-4})^{d^2}
    \ge 4d^{-4}e^{-d^{-2}}
    \ge d^{-4}=:q>\mathbb{P}(F_v),\quad\text{ for every $v\in V(G)$}.
\end{align*}
By Corollary \ref{cor: lll}, we obtain that with probability at least $\frac{1}{2}$, there are sets $A_1,\ldots, A_k$ which satisfy the following for every $i \in \br{k}$. 
\begin{itemize}
    \item $\left|S_i\triangle A_i\right|=O(n/d^3)=o_d(n/k)$.
    \item $d(v,A_i)\in \left[\frac{\delta d}{k}+1,\frac{Cd}{k}-1\right]$ for every $v \in V(G)$.
\end{itemize}
We now fix these sets. Due to the first item and the Chernoff bound, we also get $\left||A_i|-\frac{n}{k}\right|=O(n/d^{3})$ \textbf{whp}.

We now want to `balance the sets', that is, obtain sets $V_1,\ldots, V_k$, each of size exactly $\frac{n}{k}$, which satisfy that, for every $i\in\br{k}$, $\left|S_i\triangle V_i\right|=o_d(n/k)$ and $d(v,V_i)\in \left[\frac{\delta d}{k}+1,\frac{Cd}{k}-1\right]$ for every $v\in V(G)$. To that end, we follow the same proof as in Lemma \ref{l: final round} --- we iteratively move vertices from sets $A_i$ of size larger than $\frac{n}{k}$ into sets $A_j$ of size smaller than $\frac{n}{k}$, each time choosing a vertex which is not in the second neighbourhood of the previous vertices (and thus the degree of any vertex does not change by more than one into any of the sets), and utilise that $d^2\cdot O(n/d^3)=o_d(n/k)$. After this procedure, we obtain the required sets.
\end{proof}

\section{Growing the trees}\label{s: theorems proof}
In this section, we show how to construct the tree factor, given the vertex partition guaranteed by Propositions \ref{prop: main path} and \ref{prop: main small k}. 

In Section \ref{subsection:typical-properties}, we collect typical properties of $G(n, d)$ which are important for proving Theorem \ref{th: main}. In Section \ref{subsection:proof-of-main-thm}, we prove Theorem~\ref{th: main}. We do so by applying Proposition \ref{prop: main path} (or its stronger version, Proposition \ref{prop: main small k}, when trees are small) and, using the typical properties we have shown in Section \ref{subsection:typical-properties}, we will find a perfect matching between $V_i$ and $V_j$ for every $\{i,j\} \in E(T)$.

\subsection{Typical Properties of random regular graphs}\label{subsection:typical-properties}
The following claim shows that typically there are not `too many' edges between any two small sets of equal size. 
\begin{claim}\label{claim:global-event-between-small-sets}
    For every $\epsilon, \delta > 0$, there exists $\eta > 0$ such that, for sufficiently large $d$, \textbf{whp} the following holds in $G\sim G(n, d)$. Let $k \le \frac{(1-\epsilon)d}{\log d}$ be an integer. Then, for every two disjoint sets $A,B\subseteq V(G)$ satisfying $|A| = |B| < \eta \cdot \frac{n}{k}$,
    \[
        e(A, B) < |A| \cdot \delta \cdot \frac{d}{k}.
    \]    
\end{claim}
\begin{proof}
    We have at most $n$ choices for the sizes of the sets $A$ and $B$. For a fixed size $a$, we have at most $\binom{n}{a}^2$ choices for the sets of such size. By the union bound over the (at most) $n$ values of $a$ and the number of choices of the sets, it suffices to prove that, for every fixed $A$ and $B$ satisfying $|A| = |B|\eqqcolon a < \eta \cdot \frac{n}{k}$, we have
    \[
        \Pr\left(e(A, B) \ge a \cdot \delta \cdot \frac{d}{k}\right) = o\left(\frac{1}{n \binom{n}{a}^2}\right)\quad\text{ uniformly in $a$.}
    \]   

    Fix $a < \eta \cdot \frac{n}{k}$ and set $\ell = a \cdot \delta \cdot \frac{d}{k}$. Note that $2\ell=\frac{2a\delta d}{k}<\frac{2\eta \delta n d}{k^2}<\delta nd$, and thus $M-2\ell\ge \frac{nd}{3}=\omega(d^2).$
    Fix two disjoint sets $A$ and $B$ of size $a$. Note that $d(A) = a d\le \eta \cdot \frac{nd}{k}<nd/2=M/2$ where $M \coloneqq \sum_{v} d(v)$ (and the second inequality holds for $\eta>0$ sufficiently small). Thus, we may apply Corollary \ref{cor: gao} and obtain,
    \begin{align*}
        \Pr(e(A, B) \ge \ell) &\le \binom{d(A)}{\ell} \left(\frac{d(B)}{M(1 + o(1))}\right)^\ell \le \left(\frac{e a d}{\ell}\right)^\ell \left(\frac{ad}{nd(1+o(1))}\right)^\ell = \left(\frac{e a^2 d}{\ell n(1+o(1))}\right)^\ell. 
    \end{align*}
    Thus,
    \begin{align*}
        \Pr(e(A, B) \ge \ell) \cdot n \binom{n}{a}^2 &\le \left(\frac{e a^2 d}{\ell n (1+o(1))}\right)^\ell \cdot n \binom{n}{a}^2 \le \left(\frac{e a^2 d}{\ell n(1+o(1))}\right)^\ell \cdot n \left(\frac{en}{a}\right)^{2a} \\
        &\le \left(\frac{2 e a k}{\delta  n}\right)^{a \cdot \delta \cdot \frac{d}{k}} \cdot n \left(\frac{en}{a}\right)^{2a} \\        
        &= \exp\left\{\log n + a\left[2 \log\left(\frac{en}{a}\right) - \delta \cdot \frac{d}{k} \cdot \log\left(\frac{\delta  n}{2eak}\right)\right]\right\}.
    \end{align*}
    Observe that the function $f(a) = 2 \log\left(\frac{en}{a}\right) - \delta \cdot \frac{d}{k} \cdot \log\left(\frac{\delta  n}{2eak}\right)$ is increasing for all $a > 0$ whenever $d$ is large enough. Indeed,
    \begin{align*}
        f'(a) = -\frac{2}{a} + \frac{\delta d}{a k} \ge \frac{1}{a}\left(\frac{\delta \cdot \log d}{(1-\epsilon)} - 2\right) > 0,        
    \end{align*}
    where the first inequality is true since $k \le \frac{(1-\epsilon)d}{\log d}$. Moreover, we have that 
    \begin{align*}
        f\left(\eta \cdot \frac{n}{k}\right) &= 2 \log \left(\frac{e k}{\eta}\right) - \delta \cdot \frac{d}{k} \cdot \log\left(\frac{\delta}{2 e \eta}\right) \\
        &\le 2\log(d) - \frac{\delta}{1-\epsilon} \cdot \log d \cdot \log\left(\frac{\delta}{2 e \eta}\right) < -1,
    \end{align*}        
    where the last inequality is true whenever $\eta$ is sufficiently smaller than $\delta$. Thus, if $a \le n^{1-\xi}$ for some small constant $\xi > 0$, then
    \begin{align*}
        \Pr(e(A, B) \ge \ell) \cdot n \binom{n}{a}^2 &\le \exp\left\{\log n + a \cdot f\left(n^{1-\xi}\right)\right\} \le  \exp\left\{\log n + f\left(n^{1-\xi}\right)\right\}\\
        &= \exp\left\{\log n + \left[2 \log (e n^\xi) - \delta \cdot \frac{d}{k} \cdot \log\left(n^\xi \cdot \frac{\delta}{2ek}\right)\right]\right\} \\
        &\le \exp\left\{\log n + 3 \xi \log n - 0.5 \xi \delta \cdot \frac{d}{k} \cdot \log n\right\} = o(1),        
    \end{align*}
    where the last equality is true for large enough $d$ since $\frac{d}{k} \ge \log d$.
    
    Further, if $n^{1 - \xi} \le a \le \eta \cdot \frac{n}{k}$, then
    \begin{align*}
        \Pr(e(A, B) \ge \ell) \cdot n \binom{n}{a}^2 &\le e^{\log n + a \cdot f(a)}
        \le e^{\log n - a} 
        \le e^{\log n - n^{1-\xi}}  = o(1).
    \end{align*}        
\end{proof}

Claim \ref{claim:global-event-between-small-sets} is useful in showing Hall's condition between small sets in $V_i$ and $V_j$ (that is, every small set $U \subseteq V_i$ has many neighbours in $V_j$). We would like to bound the neighbourhoods of large sets as well. This is the essence of the next claim.

\begin{claim}\label{claim:typical neighbourhood-S}      
    For every $\epsilon, \eta > 0$, there exist $\epsilon_1, \epsilon_2 > 0$ such that, for sufficiently large $d$, the following holds in $G \sim G(n, d)$. Let $2\le k \le \frac{(1-\epsilon)d}{\log d}$ be an integer and let $S_1,\ldots, S_k$ be a uniformly random partition of $V(G)$. Then, \textbf{whp}, for every $i \neq j \in \br{k}$ and $A \subset S_i$ satisfying  $\eta \frac{n}{k} \le |A| \le 0.5(1 + \epsilon_1)\frac{n}{k}$, we have
    $|N(A, S_j)| \ge (1 + \epsilon_2) |A|$.    
\end{claim}
\begin{proof}
    Let $\xi > 0$ be a sufficiently small constant. We first show that \textbf{whp}, for every $i, j \in \br{k}$ (not necessarily different), we have that $d(v,S_j)\in[(1-\xi)d/k,(1+\xi)d/k]$ for all but at most $o_d(n/k)$ of $v\in S_i$. Fix $i,j\in \br{k}$. For every vertex $v \in V(G)$, denote by $Z_v$ the indicator random variable of the event that the number of neighbours of $v$ in $S_j$ is not in the interval $[(1-\xi) d/k, (1+\xi) d/k]$. Set $Z \coloneqq \sum_{v \in V(G)} Z_v \cdot \mathbf{1}_{v \in S_i}$. By Lemma \ref{lemma:binomial-bounds}, for every vertex $v$, 
    \begin{align*}
        \Pr\left(Z_v = 1\right) &= \Pr\left(\left|\text{Bin}\left(d, \frac{1}{k}\right) - \frac{d}{k}\right| > \xi \frac{d}{k}\right)\le 2\cdot e^{-\frac{\xi^2d}{3k}}\le  d^{-\xi^2/4},
    \end{align*}
    where the last inequality follows since $\frac{d}{k} \ge \frac{\log d}{1-\epsilon}$.  
    
    Further, the events that $v \in S_i$ and $Z_v = 1$ are independent. Hence, we have \begin{equation}
        \mathbb{E}[Z] \le \frac{n}{k\cdot d^{\xi^2/4}}.
    \label{eq:E_Z}    
    \end{equation}
    For every vertex $v \in V(G)$, the random variable $Z_v \cdot \mathbf{1}_{v \in S_i}$ is independent of all but at most $d^2$ other random variables $Z_u \cdot \mathbf{1}_{u \in S_i}$. Therefore,
    \begin{align*}
        \text{Var}(Z) &= \sum_{v, u \in V(G)} \text{Cov}(Z_v \cdot \mathbf{1}_{v \in S_i}, Z_u \cdot \mathbf{1}_{u \in S_i}) \\
        &\le \sum_{v \in V(G)} d^2 \cdot \max_{u \in V(G)} \Pr\left(Z_v \cdot \mathbf{1}_{v \in S_i} = Z_u \cdot \mathbf{1}_{u \in S_i} = 1\right) \le d^2 n.
    \end{align*}
    By Chebyshev's inequality,
    \[
        \Pr\left(|Z - \mathbb{E}[Z]| > \frac{n}{d^{\xi^2/4} k}\right) = o_n(1).
    \]
    Therefore, by the union bound and due to~\eqref{eq:E_Z}, \textbf{whp} for every $i,j\in\br{k}$, we have that all but at most $o_d(n/k)$ vertices $v\in S_i$ satisfy $d(v,S_j)\in \left[(1-\xi)d/k,(1+\xi)d/k\right]$. 

    Let $C > 0$ be a large constant. Similarly to the previous argument, we can show that \textbf{whp} for every $i, j\in\br{k}$, in every $S_i$ there are $o(n/d^{100})$ vertices of degree larger than $C \cdot d/k$ into $S_j$. In addition, by a straightforward application of Lemma \ref{lemma:binomial-bounds} we have that \textbf{whp} $|S_i| = \frac{n}{k} + O(n^{0.51})$ for every $i \in \br{k}$.

    Assume that the above holds deterministically. That is, for every $i, j \in \br{k}$, we have $d(v,S_j)\in \left[(1-\xi)d/k,(1+\xi)d/k\right]$ for all but at most $o_d(n/k)$ vertices $v \in S_i$ and that $d(v, S_j) \le C \cdot d/k$ for all but at most $o(n/d^{100})$ vertices $v \in S_i$.
    
    Note that, for every $i, j \in \br{k}$ and any set $A \subseteq S_i$ of size at least $\eta n/k$, we have 
    \begin{align}\label{eq: A S_j edges}
        e(A, S_j) \ge (|A| - o_d(n/k)) \cdot (1-\xi) \frac{d}{k} > |A| \cdot (1-2\xi) \frac{d}{k},
    \end{align}            
    where the last inequality is true for sufficiently large $d$. Notice that $e(A, S_j)=e(A, N(A,S_j))$. Thus, if we show that for every choice of $B$ of size $(1+\epsilon_2) |A|$, we have that $e(A, B) \le |A| \cdot (1-2\xi) \frac{d}{k}$, then it implies that $|N(A, S_j)| \ge (1 + \epsilon_2) |A|$. Indeed, if $|N(A, S_j)| < (1 + \epsilon_2) |A|$, then we may take $N(A, S_j) \subseteq B \subseteq S_j$ to be of size $(1+\epsilon_2) |A|$ and get $e(A, S_j) = e(A, B) \le |A| \cdot (1-2\xi) \frac{d}{k}$, a contradiction to \eqref{eq: A S_j edges}.

    We now show that \textbf{whp} $e(A, B) \le |A| \cdot (1-2\xi) \frac{d}{k}$ for every $i \neq j \in \br{k}$, $A \subseteq S_i$ satisfying $\eta \frac{n}{k} \le |A| \le 0.5(1 + \epsilon_1)\frac{n}{k}$ and $B \subseteq S_j$ of size $(1 + \epsilon_2) |A|$.
    
    Fix $i \neq j \in \br{k}$, fix $\eta \frac{n}{k} \le a \le 0.5(1 + \epsilon_1)\frac{n}{k}$, fix $A \subset S_i$ of size $a$ and fix $B \subseteq S_2 $ of size $(1 + \epsilon_2)a$. We now prepare the ground for the usage of Lemma \ref{lemma:Gao-usage}.
    
    Consider the graph $\Tilde{G}$ induced by $S_i \cup S_j$. Note that $\Tilde{G}$, given its degree sequence, has a uniform distribution. Moreover, all but at most $o_d(n/k)$ vertices from $S_i \cup S_j$ have degree at least $(1-\xi) \frac{d}{k}$ to $S_i \cup S_j$. We also have $|S_i \cup S_j| = \frac{2n}{k} + O(n^{0.51})$. Hence,
    \begin{align*}
        \sum_{v \in S_i \cup S_j} d_{\Tilde{G}}(v) \ge \left(|S_i \cup S_j| -o_d(n/k)\right) \cdot 2(1-\xi) \frac{d}{k} \ge 4(1-2\xi) \cdot \frac{n}{k} \cdot \frac{d}{k} \ge 8(1-3\xi) \cdot a \cdot \frac{d}{k},
    \end{align*}
    where the last inequality is true since $a \le 0.5(1+\epsilon_1) \frac{n}{k}$ and whenever $\epsilon_1$ is sufficiently smaller than $\xi$.

    Furthermore, since all but at most $o_d(n/k)$ vertices in $S_i \cup S_j$ have degree at most $2(1+\xi) \frac{d}{k}$ into $S_i \cup S_j$ and all but at most $n/d^{100}$ vertices in $S_i \cup S_j$ have degree at most $C\cdot \frac{d}{k}$ into $S_i \cup S_j$, we have
    \begin{align*}
        d_{\Tilde{G}}(A) &\le 2(1+\xi) \cdot \frac{d}{k} \cdot |A| + 2C \cdot \frac{d}{k} \cdot  o_d\left(\frac{n}{k}\right) + 2d \cdot \frac{n}{d^{100}}
        \\& = 2(1+\xi) \cdot a \cdot \frac{d}{k} + 2C \cdot \frac{d}{k} \cdot  o_d\left(\frac{n}{k}\right) + 2d \cdot \frac{n}{d^{100}} \le 2(1+2\xi) \cdot a \cdot \frac{d}{k}.
    \end{align*}
    Similarly, we have
    \begin{align*}
        d_{\Tilde{G}}(B) &\le 2(1+\xi) \cdot \frac{d}{k} \cdot |B| + 2C \cdot \frac{d}{k} \cdot  o_d\left(\frac{n}{k}\right) + 2d \cdot \frac{n}{d^{100}}
        \\&\le  2(1+\xi) \cdot (1+\epsilon_2) a \cdot \frac{d}{k} + 2C \cdot \frac{d}{k} \cdot  o_d\left(\frac{n}{k}\right) + 2d \cdot \frac{n}{d^{100}} \le 2(1+2\xi) \cdot a \cdot \frac{d}{k},
    \end{align*}
    where the last inequality is true whenever $\epsilon_2$ is sufficiently smaller than $\xi$.

    Hence, by Lemma \ref{lemma:Gao-usage} (applied with $t = a \cdot \frac{d}{k}$),
    \begin{align*}
        \Pr\left(e(A, B) \ge (1-2\xi) a \cdot \frac{d}{k}\right) \le 0.95^{a \cdot \frac{d}{k}}\le0.95^{a\log d},
    \end{align*}
    since $k\le \frac{(1-\epsilon)d}{\log d}$. Hence, by the union bound, the probability that there exist such sets $A$ and $B$ such that $e(A, B) > |A| \cdot (1-2\xi) \frac{d}{k}$ is at most
    \begin{align*}
        d^2 \cdot \sum_{a = \eta \frac{n}{k}}^{(0.5 + \epsilon_1)\frac{n}{k}} \binom{n/k + O(n^{0.51})}{a}\binom{n/k + O(n^{0.51})}{(1+\epsilon_2)a} 0.95^{a \log d} &\le \sum_{a = \eta \frac{n}{k}}^{(0.5 + \epsilon_1)\frac{n}{k}} e^{3a \log(10n/(ak))} 0.95^{a \log d} \\
        &\le \sum_{a = \eta \frac{n}{k}}^{(0.5 + \epsilon_1)\frac{n}{k}} e^{3a \log (10/\eta)} 0.95^{a \log d} \\
        &\le n e^{-a} = o_n(1),
    \end{align*}
    where the last inequality is true if $d$ is sufficiently large.
\end{proof}

\subsection{Proof of Theorem \ref{th: main}}\label{subsection:proof-of-main-thm}
Let $\epsilon > 0$ be a constant and let $d$ be a sufficiently large integer. Let $T$ be a tree on $k\le \frac{(1-\epsilon)d}{\log d}$ vertices. We will show that \textbf{whp} $G \sim G(n, d)$ contains a $T$-factor. Let $\delta = \delta(\epsilon) > 0$ and $C = C(\epsilon) > 0$ be the constants guaranteed by Proposition \ref{prop: main path}. In addition, let $\eta = \eta(\epsilon, \delta)$ be the constant guaranteed by Claim \ref{claim:global-event-between-small-sets} and let $\epsilon_1 = \epsilon_1(\eta^2,\epsilon)$ and $\epsilon_2 = \epsilon_2(\eta^2,\epsilon)$, be the constants guaranteed by Claim \ref{claim:typical neighbourhood-S}.

Now, let $S_1,\ldots, S_k$ be such that every $v\in V(G)$ is assigned to $S_i$ for an index $i\in \br{k}$ chosen uniformly at random, independently from all the other vertices. Note that \textbf{whp} $G(n,d)\in \mathcal{G}_d$ (see, for example, \cite{W99}) and the statements of Claims \ref{claim:global-event-between-small-sets} and \ref{claim:typical neighbourhood-S}, are satisfied. We then fix a deterministic $G\in \mathcal{G}_d$ that satisfies conclusions of both claims.

Let $\Sigma$ be the set of all partitions of $V(G)$ into $k$ ordered sets. Let $\Sigma'\subset\Sigma$ be the set of all \emph{nice} $(S_1,\ldots,S_k)$, i.e. those that satisfy the conclusion of Proposition~\ref{prop: main path}. Due to Proposition~\ref{prop: main path}, there exists a constant $\gamma>0$ such that $|\Sigma'|/|\Sigma|\geq\gamma$. On the other hand, let $\Sigma''\subset\Sigma$ be the set of all good $(S_1,\ldots,S_k)$, i.e. those that satisfy the conclusion of Claim~\ref{claim:typical neighbourhood-S}. We know that $|\Sigma''|/|\Sigma|=1-o(1)$. We immediately get that there exists a tuple $(S_1,\ldots,S_k)$ which is simultaneously nice and good. Since this tuple is nice,
 there exist sets $V_1, \dots, V_k$ which satisfy all the desired requirements. Under these assumptions, we will be able to show deterministically that there exists a perfect matching between $V_i$ and $V_j$ for every $\{i,j\} \in E(T)$ which implies the existence of a $T$-factor. One way to show the latter implication is, for example, by induction on $k$. Assume without loss of generality that $k \in V(T)$ is a leaf and that, by induction assumption, we have a $T'$-factor in $\cup_{i=1}^{k-1} V_i$ where $T' = T \setminus \{k\}$. We may then complete $T'$ to a $T$-factor via the perfect matching between $V_k$ and $V_i$ where $i$ is the only neighbour of $k$ in $T$.

Fix $\{i,j\} \in E(T)$. We will show that Hall's condition is satisfied between $V_i$ and $V_j$ in $G$. Let $W \subseteq V_i$. We will prove that $|N(W, V_j)| \ge |W|$. By Proposition \ref{prop: main path}, for every $v \in V_i$, we have $d(v, V_j) \in \left[\delta \cdot \frac{d}{k}, C \cdot \frac{d}{k}\right]$. Hence, 
\begin{align}\label{eq:W-A_j-edges}
    e(W, V_j) \ge |W| \cdot \delta \cdot \frac{d}{k}.
\end{align}
We split the proof into three parts depending on the size of $|W|$.

First of all, we show that if $|W| < \eta \cdot \frac{n}{k}$, then $|N(W, V_j)| > |W|$. Assume towards contradiction that this is false. Then, there exists a set $B \subseteq V_j$ satisfying $N(W, V_j) \subseteq B$ and $|B| = |W|$. By Claim \ref{claim:global-event-between-small-sets}, we have $e(W, B) = e(W, V_j) < |W| \cdot \delta \cdot \frac{d}{k}$, a contradiction to \eqref{eq:W-A_j-edges}.

Next, assume that $\eta \cdot \frac{n}{k} \le |W| \le 0.5(1 + \epsilon_1) \frac{n}{k}$. We have
\begin{align*}        
    |W \cap S_i| \ge |W| - |V_i \setminus S_i| \ge \eta \cdot \frac{n}{k} - |V_i \setminus S_i| \ge \eta^2 \cdot \frac{n}{k},
\end{align*}
where the last inequality is true since $|V_i \setminus S_i| = o_d(n/k)$ by Proposition \ref{prop: main path}. Thus,
\begin{align*}
    |N_{V_j}(W)| &\ge |N_{S_j}(W \cap S_i) \cap V_j| \ge (1 + \epsilon_2) |W \cap S_i| - |V_j \setminus S_j| \\
    &\ge (1 + \epsilon_2)(|W| - |V_i \setminus S_i|) - |V_j \setminus S_j| > |W|,
\end{align*}
where the second inequality is true by Claim \ref{claim:typical neighbourhood-S} and the last inequality is true since $|V_i \setminus S_i|, |V_j \setminus S_j| = o_d(n/k)$ and $|W| = \Omega(n/k)$.

Finally, assume that $0.5(1 + \epsilon_1) \frac{n}{k} < |W| \le \frac{n}{k}$. Assume towards contradiction that $|N(W, V_j)| < |W|$. Let $B \subseteq V_j \setminus N(W, V_j)$ be an arbitrary set of size $|V_i \setminus W|$. Notice that $N(B, V_i) \subseteq V_i \setminus W$, otherwise there exists $v \in B$ which is adjacent to $u \in W$. This in turn implies that $v \in N(W, V_j)$ and, in particular, $v \notin B$ --- a contradiction. Moreover, 
\[
    |B| = |V_i| - |W| < |V_i| - 0.5(1 + \epsilon_1) \frac{n}{k} = 0.5(1-\epsilon_1) \frac{n}{k}.
\]
Therefore, by the previous argument (with $i$ and $j$ reversed), $|N(B, V_i)| > |B|$. However, since $N(B, V_i) \subseteq V_i \setminus W$, we have $|N(B, V_i)| \le |V_i \setminus W| = |B|$ --- contradiction.


\paragraph{Acknowledgements} The authors thank Itai Benjamini for bringing the question to our attention. The authors wish to thank Michael Krivelevich and Noga Alon for fruitful discussions. In particular, we thank Michael Krivelevich for showing us the alternative argument from~\cite{DK}, after the first version of this paper was uploaded. It allowed us to improve the presentation significantly.

\bibliographystyle{abbrv}
\bibliography{sat}

\appendix
\section[d-1-star-factor]{$K_{1,d-1}$-factor}
\label{appendix}
Let us show that, for $d\ge 5$, $G(n,d)$ \textbf{whp} does not contain a $K_{1,d-1}$-factor. Setting $N=nd$, $M=\frac{n}{d}+\frac{n(d-1)^2}{d}=N-2n+2n/d$, we have that the expected number of graphs that correspond to $K_{1,d-1}$-factors in the configuration model (see, for example, ~\cite{Bol_book}) is at most
\begin{align*}
    \frac{\binom{n}{n/d}d^{n/d}(n\frac{d-1}{d})!d^{n(d-1)/d}M!/(2^{M/2}(M/2)!}{N!/(2^{N/2}(N/2)!)}&=\left(\sqrt{d}+o(1)\right)\frac{n^nd^n(M/N)^{M/2}e^{n-n/d}}{(n/d)^{n/d}e^{n(d-1)/d}(nd)^{n-n/d}}\\
    &=\left(\sqrt{d}+o(1)\right)\left(d^{2/d}\left(1-\frac{2}{d}+\frac{2}{d^2}\right)^{d/2-1+1/d}\right)^n\\
    &=o(1).
\end{align*}
Indeed, $g(d)=d^{2/d}\left(1-\frac{2}{d}+\frac{2}{d^2}\right)^{\frac{d}{2}-1+\frac{1}{d}}$ decreases in $d$ on $[5,\infty)$, and $g(5)<1$.
\end{document}